\documentclass[oneside,notitlepage,12pt]{article}

\usepackage{amssymb}
\usepackage[leqno]{amsmath}
\usepackage{amsfonts}
\usepackage{amsopn}
\usepackage{amstext}
\usepackage{amsthm}
\usepackage[normalem]{ulem}

\usepackage{tikz}
\usetikzlibrary{arrows.meta}
\usetikzlibrary{graphs}
\usetikzlibrary{decorations.pathmorphing}

\usepackage{enumitem}

\usepackage{verbatim}
\usepackage[colorlinks, backref]{hyperref}

\usepackage{mathrsfs}
\usepackage{calrsfs}
\usepackage{clock} 

\usepackage{marvosym} 

\setlist{itemsep = 0pt}
\setlist[enumerate, 1]{label=\upshape (\arabic*), ref=(\arabic*)}
\setlist[enumerate, 2]{label=\upshape (\arabic{enumi}\alph*), ref=(\arabic{enumi}\alph*)}

\newtheorem{tw}{Theorem}[section]
\newtheorem{wn}[tw]{Corollary}
\newtheorem{lm}[tw]{Lemma}
\newtheorem{prop}[tw]{Proposition}

\theoremstyle{definition}
\newtheorem{df}[tw]{Definition}

\newtheorem{question}[tw]{Question}
\theoremstyle{remark}
\newtheorem{uwgi}[tw]{Remark}
\newtheorem{observation}[tw]{Observation}
\newtheorem{strategy}[tw]{Strategy}

\newcommand{\green}[1]{{\color{green}{#1}}}

\newcommand \set [1]{\{#1\}}
\newcommand \seq [1]{\langle #1 \rangle}
\newcommand \map [3]{#1\colon #2 \to #3} 
\newcommand \maps {\colon} 
\newcommand \im [1]{[#1]} 

\newcommand \from {\leftarrow} 
\newcommand \restr [1] {\mathord{\upharpoonright}_{#1}}
\newcommand \id [1]{{\operatorname{id}_{#1}}}
\newcommand \dom {\operatorname{dom}}
\newcommand \rng {\operatorname{rng}}

\newcommand \cmp {\circ} 
\newcommand \subs {\subseteq}

\newcommand \card [1] {\lvert#1\rvert} 

\newcommand \sig {\sigma}
\newcommand \eps {\varepsilon}

\newcommand \N {\mathbb{N}}
\newcommand \Z {\mathbb{Z}}
\newcommand \Q {\mathbb{Q}}
\newcommand \R {\mathbb{R}}
\newcommand \Qpos {\Q^{\geq 0}}
\newcommand \Rpos {\R^{\geq 0}}

\newcommand \U {\mathbb{U}}
\newcommand \Su {\mathscr{S}}
\newcommand \C {\mathscr{C}}
\newcommand \D {\mathscr{D}}
\newcommand \K {\mathscr{K}}

\newcommand \fin {{\operatorname{fin}}}

\newcommand \Lin {\mathfrak{L}}
\newcommand \Ult {\mathfrak{U}}
\newcommand \UltIso {\mathfrak{I}}
\newcommand \UltConv {\Ult^\prec}

\newcommand \tp {{\rm{tp}}}

\newcommand \Aut {\operatorname{Aut}}
\newcommand \Iso {\operatorname{Iso}}

\let \aut \Aut

\newcommand \email [1] {%
    \quad{\Letter\,\small\href{mailto:#1}{\nolinkurl{#1}}}%
}

\newcommand \arxivlink [1]{%
    \href{https://arxiv.org/abs/#1}{\texttt{arXiv:#1}}%
}

\hypersetup{
	pdftitle = {Universal homogeneous two-sorted ultrametric spaces},
	pdfauthor = {Adam Bartoš, Wiesław Kubiś, Aleksandra Kwiatkowska, Maciej Malicki},
	colorlinks,
	linkcolor = [rgb]{0, 0, 0.7},
	citecolor = [rgb]{0, 0.7, 0},
	urlcolor = [rgb]{0.4, 0.2, 0},
	linktoc = section,
}

\title{Generic dc-automorphisms of two-sorted ultrametric spaces}
\author{
{\sc Adam Barto\v{s}}
    \email{bartos@math.cas.cz}\\
    {\small Institute of Mathematics, Czech Academy of Sciences (CZECHIA)}
\and
{\sc Wies{\l}aw Kubi\'s}
    \email{kubis@math.cas.cz}\\
    {\small Institute of Mathematics, Czech Academy of Sciences (CZECHIA)}
\and
{\sc Aleksandra Kwiatkowska}
    \email{aleksandra.kwiatkowska@math.uni.wroc.pl}\\
    {\small University of Wrocław (POLAND) and  University of M\"unster (GERMANY)}
\and
{\sc Maciej Malicki}
    \email{mmalicki@mimuw.edu.pl}\\
    {\small Faculty of Mathematics, Informatics and Mechanics, University of Warsaw (POLAND)}
}

\date{\clocktime\today}

\begin{document}

\maketitle

\begin{abstract}
We continue to study ultrametric spaces as two-sorted structures consisting of a set
of points and of a linearly ordered set of distances, together with the dc-embeddings, which we introduced in our earlier paper ``Universal homogeneous two-sorted ultrametric spaces''. The class of all finite two-sorted ultrametric spaces with dc-embeddings is Fraïssé whose limit we denote by $\U$.

The main result of the article is  that $\Aut(\U)$ has a comeager conjugacy class. For that we show the cofinal amalgamation property of partial automorphisms and characterize amalgamation bases. In fact we develop a general strategy for showing cofinal amalgamation property for a broad class of categories.

Furthermore, we show that there is no generic pair of automorphisms, we provide a detailed description of single orbits under dc-automorphisms, and we prove that any finite partial dc-automorphism, even in the presence of other
orbits, can be extended to one that is closed or monotone. 
\end{abstract}


\vspace{-2ex}

\color{black}

\vspace{-2ex}
\paragraph{Acknowledgments.} 
Research of Adam Bartoš and Wiesław Kubiś was supported by GAČR (Czech Science Foundation) grant EXPRO 20-31529X and by the Czech Academy of Sciences (RVO 67985840). Research of Aleksandra Kwiatkowska was supported by DFG (German Research Foundation) under Germany’s Excellence Strategy EXC 2044–390685587, Mathematics M\"{u}nster: Dynamics–Geometry–Structure. Research of Maciej Malicki was partially supported by the National Science Centre, Poland
under the Weave-UNISONO call in the Weave programme [grant no
2021/03/Y/ST1/00072].

\tableofcontents

\section{Introduction}

A metric space $(X,d)$ is called an {\em ultrametric space} if for all $x,y,z\in X$, the ultrametric triangle inequality holds: $d(x,y)\leq\max\{d(x,z), d(y,z)\}$. 
Consequently, every triangle is isosceles.
The {\em (open) ball} of center $x\in X$ and radius $r>0$ in an ultrametric space is the set $B_r(x)=\{y\in X\colon d(y,x)<r\}$.
Note that for any two balls $B$ and $C$, they are either disjoint ($B\cap C=\emptyset$) or one is contained in the other ($B\subseteq C$ or $C\subseteq B$). Moreover, every point in a ball can serve as its center.

In our previous paper~\cite{paper1}, we reframed ultrametric spaces as two-sorted structures consisting of a set of points and a linearly ordered set of distances, in order to construct a universal separable ultrametric space, which is impossible in the classical setting.
Instead of isometric embeddings, we used appropriate two-sorted embeddings called \emph{dc-embeddings} (for “distance-carrying”) as morphisms.
Furthermore, we initiated a systematic study of two-sorted ultrametric spaces, ways of representing them, methods of constructing universal objects, and properties of their automorphism groups.
In particular, we studied the generic countable two-sorted ultrametric space~$\U$.
We showed that the class $\Ult_\fin$ of all finite two-sorted ultrametric spaces is Fraïssé, and that its limit $\U$ can be identified with the rational Urysohn ultrametric space, the Fraïssé limit of the class $\mathfrak{C}_\fin$ of all classical finite rational ultrametric spaces.
Hence, $\U$ is homogeneous both with respect to isometries and dc-automorphisms.
We denote the respective Polish groups by $\Iso(\U)$ and $\Aut(\U)$.

In the present paper we are interested in the existence of generic $n$-tuples of dc-automorphisms of $\U$.
They behave quite differently than isometries.
Unlike an isometry, a dc-automorphism can move distances around.
In fact, we showed~\cite[Theorem~4.4]{paper1} that $\Aut(\U)$ is canonically isomorphic to the topological semidirect product $\Iso(\U) \rtimes \Aut(\Qpos)$, involving the automorphism group of the linear order on the rationals.

The situation with generic isometries is settled – by the result of Malicki~\cite[Corollary~6.1]{Mal}, every Polish Urysohn ultrametric space has ample generics of isometries.
The argument applies to the countable rational Urysohn ultrametric space as well – by \cite[Lemma~4.2]{Mal}, the class $\mathfrak{C}_\fin$ has the \emph{Hrushovski property} (EPPA), i.e.  every $A \in \mathfrak{C}_\fin$ is contained in $B \in \mathfrak{C}_\fin$ such that every partial isometry of $A$ extends to a total isometry of $B$.
Since ultrametric spaces have a canonical amalgamation (as in Lemma~\ref{thm:complete_isometric_amalgamation}), $\mathfrak{C}_\fin$ has also \emph{amalgamation property with automorphisms} (APA), i.e. all $A \leq B, C$ in $\mathfrak{C}_\fin$ admit an amalgamation $D \in \mathfrak{C}_\fin$ such that all coherent isometries $f \in \Iso(A)$, $g \in \Iso(B)$, and $h \in \Iso(C)$ extend to an isometry of $D$.
The Fraïssé limit of a class with EPPA and APA automatically has ample generics (see \cite[Theorem~2.1.5]{Siniora}).

This standard strategy for obtaining ample generics completely fails for dc-automorphisms – since finite linear orders have no nontrivial automorphisms, every finite total dc-automorphism is an isometry, and so $\Ult_\fin$ does not have EPPA.
In fact, we show that $\U$ does not have a generic pair of dc-automorphisms.
On the other hand, a generic dc-automorphism exists, but we use a strategy based on so-called \emph{determined partial automorphisms} in place of finite total automorphisms.

Let us summarize the results obtained in the article.
\begin{itemize}
    \item We provide a detailed description of single orbits under dc-automorphisms  (Theorem~\ref{thm:single_orbit}), which extends to partial dc-automorphisms (Corollary~\ref{thm:partial_single_orbit}).
    Any orbit of a finite partial dc-automorphism, even in the presence of other orbits, can be extended to one that is closed or monotone (Theorem~\ref{closedormonotone}).
    
    \item There is a generic dc-automorphism of $\U$ (Corollary~\ref{cor:generic_auto}) but no generic pair (Corollary~\ref{cor:no_generic_pair}).
    Moreover, every orbit of the generic dc-automorphism is either closed or monotone, and the induced automorphism on distances is the generic automorphism of $\Qpos$ (Corollary \ref{wn:generic-clo or mon}).

    \item Similarly, there is no generic pair of automorphisms of the Urysohn ultrametric space $\U^\prec$ endowed with the generic convex order (Corollary~\ref{cor:no_generic_pair_convex}), and we believe $\U^\prec$ has a generic automorphism (Question~\ref{que:generic_auto_convex}).

    \item To show the existence of a generic dc-automorphism of $\U$ we prove the cofinal amalgamation property (CAP) of the category $\Ult^*_\fin$ of finite partial dc-automorphisms, and we characterize amalgamation bases (Theorem~\ref{spaces_totaliz}).

    \item To that end we develop a general strategy (Strategy~\ref{strategy}) for showing CAP of $\C^*_\fin$ for any so-called nice category of structures $\C$.
    For such a category, every amalgamation base in $\C^*_\fin$ is a determined partial automorphism -- the partial orbits can be completed to total ones in a unique way (Lemma~\ref{thm:amalgamation_base_determined}).
    The strategy is summarized as follows:
    \begin{itemize}
        \item First show that total (potentially infinite) automorphisms have AP.
        Then amalgamation bases are exactly determined partial automorphisms (Theorem~\ref{thm:amalgamation_from_totalization})
        \item Pick a full subcategory $\Su \subs \C^*_\fin$ of so-called sufficient partial automorphisms and show that they have unique one-point orbit extensions, that the extensions lie in $\Su$, and that $\Su \subs \C^*_\fin$ is cofinal.
        Then $\Su$ is a family of determined partial automorphism witnessing CAP (Proposition~\ref{prop: suffimpldeterm}).
        \item If being in $\C$ is also necessary for being an amalgamation base, we have a characterization.
    \end{itemize}
    In the cases of linear orders and ultrametric spaces, the amalgamation bases are exactly partial automorphisms whose orbits are stable, i.e. we cannot close an orbit or merge two orbits.
    We guess this choice of $\Su$ does not work in general (Question~\ref{que:stable_always_works}).
\end{itemize}

\section{Preliminaries}

\subsection{General framework}

We shall use the general framework of categories of structures introduced in \cite{paper1}.
Recall that a category $\C$ consists of a family of objects (sometimes called $\C$-objects) and of a composable family of morphisms (sometimes called arrows or $\C$-maps).
We also have the identity morphism $\id{X}\maps X \to X$ for every $\C$-object $X$.
A subcategory $\D \subseteq \C$ is called \emph{full} if every $\C$-map between $\D$-objects is a $\D$-map.
We will use some other standard category-theoretic notions such as functors and colimits. See a standard reference such as \cite{MacLane} for details.

Our categories consist of model-theoretic structures in a fixed signature (sometimes many-sorted, but we use only finitely many sorts), and of all embeddings.
We call them \emph{categories of structures}.
For every category of structures $\C$, the full subcategory of all finite structures is denoted by $\C_\fin$.
Also, there is a notion of a substructure, i.e. of an embedding that is an inclusion on the level of sets.
The notation is $X \leq Y$, meaning that $X$ is a substructure of $Y$ and that $Y$ is an extension of $X$.
Arguments are sometimes simplified by replacing embeddings with substructures. This is without loss of generality because, up to isomorphism, every embedding is an inclusion.

Recall that in general an {\it isomorphism} in a category $\C$ is a $\C$-map $f\maps X \to Y$ such that there is $\C$-map $f^{-1}\maps Y \to X$ satisfying $f^{-1} \cmp f = \id{X}$ and $f \cmp f^{-1} = \id{Y}$.
In a category of structures, it is equivalently a $\C$-map that is bijective (in every sort) as $\C$-maps are embeddings.
Isomorphisms $X \to X$ are called {\it automorphisms}.
For every $\C$-object $X$ the automorphism group of $X$ is denoted by $\Aut(X)$.

\medskip

Let $\C$ be a category of finite structures.
For every $\C$-object $A$ a \emph{partial automorphism} of $A$ is a $\C$-isomorphism $p$ such that $\dom(p), \rng(p) \leq A$, i.e. an isomorphism between substructures of $A$.
An embedding of partial automorphisms $f\maps \seq{A,p} \to \seq{B,q}$ is a $\C$-map $f\maps A \to B$ such that $f \cmp p = q \cmp f$ on $\dom(p)$, from which it follows that $f$ restricts to (unique) $\C$-maps $\dom(p) \to \dom(q)$ and $\rng(p) \to \rng(q)$.
We denote the category of partial automorphisms in $\C$ and their embeddings by $\C^*$.
We can view $\C^*$ as a category of structures. Indeed, if structures in $\C$ are one-sorted, we turn $\seq{A,p}$ into a three-sorted structure with sorts $A,\dom(p),\rng(p)$, and we add the embeddings  $\dom(p)\to A$ and $\rng(p) \to A$ as well as both $p$ and $p^{-1}$ as operations. If $\C$ is many-sorted, we do this for every sort.

More generally, for an $n$-tuple   $\seq{A, p_1, \ldots, p_n}$ such that every $p_i$ is a partial automorphism of $A$,  an embedding $f\maps \seq{A, p_1, \ldots, p_n} \to \seq{B, q_1, \ldots, q_n}$ is a $\C$-map $f\maps A \to B$ that is an embedding  from $\seq{A,p_i}$ to $ \seq{B,q_i}$ for every $i$.
The corresponding category of $n$-tuples of partial automorphisms in $\C$ is denoted by $\C^{*n}$.

\begin{observation}
Every functor $F\maps \C \to \D$ between categories of finite structures lifts to a functor $F^{*n}\maps \C^{*n} \to \D^{*n}$ as follows.

For every partial automorphism $p$ of a $\C$-object $A$ we have that $F(p)$ is a $\D$-isomorphism between $F(\dom(p)) \leq F(A)$ and $F(\rng(p)) \leq F(A)$, so $F(p)$ is a partial automorphism of $F(A)$.
Hence, it is correct to define $F^{*n}(\seq{A, p_1, \ldots, p_n}) = \seq{F(A), F(p_1), \ldots, F(p_n)}$.
(Technically, since $F$ is an abstract functor, it is not guaranteed that $F(\dom(p))$ and $F(\rng(p))$ are substructures of $F(A)$, only that they are embedded to $F(A)$.
But then, they can be canonically replaced by their images.)

Moreover, for every $\C^*$-map $f\maps \seq{A, p} \to \seq{B, q}$,  the $\D$-map $F(f)\maps F(A) \to F(B)$ restricts to  $\D$-maps $\dom(F^*(p)) \to \dom(F^*(p))$ and $\rng(F^*(p)) \to \rng(F^*(p))$.
Hence, for every $\C^{*n}$-map $f\maps \seq{A, p_1, \ldots, p_n} \to \seq{B, q_1, \ldots, q_n}$ putting $F^{*n}(f) = F(f)$ is well defined.
\end{observation}

For a category of structures $\C$ and a partial automorphism $\seq{A, p} \in \C^*$ we consider the \emph{orbit equivalence relation} (on every sort separately) defined by $a \sim_p b$ if $p^k(a) = b$ for some $k \in \Z$.
An \emph{orbit} is an equivalence class of the orbit equivalence relation.
In particular, \emph{trivial orbits} $\set{a}$ for $a \notin \dom(p) \cup \rng(p)$ are allowed.
The orbit of $a$ is sometimes denoted by $p^\Z(a)$.
An orbit $p^\Z(a)$ is called \emph{closed} if $p^k(a) = a$ for some $k > 0$.

We say that a category of structures $\K$ is \emph{hereditary} if for every $A \in \K$ and every substructure $B \leq A$ we have $B \in \K$.
Given $\seq{A, p} \in \K^*$ and a substructure $B \leq A$, we consider the \emph{restriction} $\seq{B, q} = \seq{A, p}\restr{B}$ ($q$ may also be denoted by $p\restr{B}$) where $\dom(q) = \set{b \in B: p(b) \in B}$ and $\rng(q) = \set{b \in B: p^{-1}(b) \in B}$.
If $\K$ is hereditary, then $\dom(q), \rng(q) \in \K$, and so $\seq{A, p}\restr{B} \in \K^*$.

\subsection{Fraïssé theory}

We shall recall notions from Fraïssé theory and adapt them to our context.
We follow the setup used in the previous paper~\cite{paper1}.
Here we additionally consider \emph{weak} Fraïssé classes and \emph{weak} Fraïssé sequences in order to apply Fraïssé theory to categories of partial automorphisms where the amalgamation property can hold only in trivial situations.

Let $\K$ be a category of structures.
Fraïssé theory relates properties of the class $\K_\fin$ of finite structures and of countable structures that are colimits of countable direct sequences from $\K_\fin$, i.e. up to isomorphism they are unions of countable increasing chains.
We denote this class by $\sig\K_\fin$.
We assume that $\K$ is \emph{$\sigma$-complete} in the sense that it is closed under countable unions and their isomorphic copies (i.e. under colimits of countable sequences taken in the ambient category of all structures of a given many-sorted signature).

A sequence $(A_n, f_n)$ in $\K$ consists of $\K$-objects $A_n$ and $\K$-maps $f_n\maps A_n \to A_{n + 1}$, $n \in \N$.
We also introduce notation for compositions of the bonding maps: $f_m^n\maps A_m \to A_n$ for $m \leq n \in \N$.
The colimit is denoted by $(A_\infty, f^\infty_n)$, i.e. it consists of a $\K$-object $A_\infty$ and of $\K$-maps $f_n^\infty\maps A_n \to A_\infty$ forming a cone.

Below we summarize the relevant properties of $\K_\fin$ such as variants of the amalgamation, properties of objects from $\sig\K_\fin$ such as homogeneity and extension property, and we formulate a general Fraïssé theorem in our context.
As a rule of thumb, a weak Fraïssé-theoretic property involves an extra ``guardian'' embedding $i\maps A \to A'$ through which a general embedding considered has to factorize.

We define the following properties for $\K_\fin$.
\begin{itemize}
    \item[(JEP)] The {\it joint embedding property}: for any $A,B\in\mathcal{K}_\fin$ there is $C\in \mathcal{K}_\fin$, which embeds both $A$ and $B$.
    
    \item[(AP)] The {\it amalgamation property}: for every $A\in \mathcal{K}_\fin$, $\alpha_1\colon A\to B$ and $\alpha_2\colon A\to C$ in $\mathcal{K}_\fin$ there is $D\in \mathcal{K}_\fin$ together with $\beta_1\colon B\to D$ and $\beta_2\colon C\to D$ such that $\beta_1\circ\alpha_1=\beta_2\circ \alpha_2$.
    In that case, we say that $A$ is an {\it amalgamation base} in $\mathcal{K}_\fin$.
    
    \item[(SAP)] The {\it strong amalgamation property} is a strengthening of the AP, where we additionally have $\beta_1(B)\cap \beta_2(C)=\beta_1(\alpha_1(A)) (=\beta_2(\alpha_2(A))) $.  In that case, we say that $A$ is a {\it strong amalgamation base}. 
    \item[(CAP)] The {\it cofinal amalgamation property}:  any $A\in \mathcal{K}_\fin$ embeds into $A'\in \mathcal{K}_\fin$ that is an amalgamation base in $\mathcal{K}_\fin$.
    
    \item[(WAP)] The \emph{weak amalgamation property}:
    for every $A_0 \in \K_\fin$ there is an embedding $i\maps A \to A'$ in $\K_\fin$ that is an \emph{amalgamable map}, i.e. for any $\alpha_1\colon A' \to B$ and $\alpha_2\colon A'\to C$ in $\mathcal{K}_\fin$ there is $D\in \mathcal{K}_\fin$ together with $\beta_1\colon B\to D$ and $\beta_2\colon C\to D$ such that $\beta_1\circ \alpha_1\circ i=\beta_2\circ \alpha_2\circ i$ holds.
\end{itemize}

We say that $\K_\fin \neq \emptyset$ with countably many isomorphism types is a \emph{Fraïssé class} if it satisfies JEP and AP, and it is a \emph{weak Fraïssé class} if it satisfies JEP and WAP.

Note that we do not require $\K_\fin$ to be hereditary.

\medskip

\begin{df} \label{def:Fraisse}
We say that a structure $U \in \K$ 
\begin{itemize}
    \item is \emph{universal} for $\K_\fin$  if for every $A \in \K_\fin$ there is an embedding $f\maps A \to U$,
    
    \item is \emph{homogeneous} for $\K_\fin$ if for every $A, B \in \K_\fin$ such that $A, B \leq U$ and every isomorphism $f\maps A \to B$ there is $h \in \Aut(U)$ extending $f$, or equivalently, for every $A \in \K_\fin$ and all embeddings $f, g\maps A \to U$ there is $h \in \Aut(U)$ such that $h \cmp g  = f$,
    
    \item is \emph{weakly homogeneous} for $\K_\fin$ if for every $A \in \K_\fin$ and $f\maps A \to U$ there is $A' \in \K_\fin$ and embeddings $i\maps A \to A'$ and $f'\maps A' \to B$ with $f' \cmp i = f$ such that for every $g\maps A' \to U$ there is an isomorphism $h\maps U \to U$ with $h \cmp g \cmp i = f' \cmp i$,
    
    \item has the \emph{extension property} / is \emph{injective} for $\K_\fin$ if for every $A, B \in \K_\fin$ and embeddings $f\maps A \to U$ and $g\maps A \to B$ there is an embedding $h\maps B \to U$ such that $h \cmp g = f$,
    
    \item has the \emph{weak extension property} / is \emph{weakly injective} for $\K_\fin$ if for every $A \in \K_\fin$ and $f\maps A \to U$ there is $A' \in \K_\fin$ and an embedding $i\maps A \to A'$ such that for every $B \in \K_\fin$ and $g\maps A' \to B$ there is an embedding $h\maps B \to U$ such that $h \cmp g \cmp i = f$.
\end{itemize}
Similarly we say that a sequence $(A_n, f_n)$ in $\K_\fin$ 
\begin{itemize}
    \item is \emph{cofinal} if for every $A \in \K_\fin$ there is an embedding $f\maps A \to A_n$ for some $n$,
    
    \item has the \emph{extension property} if for every $A, B \in \K_\fin$ and embeddings $f\maps A \to A_m$ and $g\maps A \to B$ there is an embedding $h\maps B \to A_n$ for some $n \geq m$ such that $h \cmp g = f^n_m \cmp f$,
    
    \item is \emph{absorbing} if for every $B \in \K_\fin$ and an embedding $g\maps A_m \to B$ there is an embedding $h\maps B \to A_n$ for some $n \geq m$ such that $h \cmp g = f^n_m$, i.e. it has the extension property for $f = \id{A_m}$,

    \item is \emph{weakly absorbing} if for every $m$ there is $m' \geq m$ such that for every $B \in \K_\fin$ and an embedding $g\maps A_{m'} \to B$ there is an embedding $h\maps B \to A_n$ for some $n \geq m$ such that  $h \cmp g \cmp i = f^n_m = f^n_{m'} \cmp i$ where $i := f^{m'}_m$,
    
    \item is \emph{Fraïssé} if it has the extension property (which under AP is equivalent to just being absorbing) and is cofinal,

    \item is \emph{weak Fraïssé} if it weakly absorbing and cofinal.
\end{itemize}
\end{df}

Now that we have all the definitions, the core of (weak) Fraïssé theory can be summarized by the following two theorems.
The Fraïssé case including references is already covered by \cite[Theorems~3.2 and 3.3]{paper1}.
The weak Fraïssé case was proved by Kubiś~\cite[Theorems~3.7, 4.2 and 5.1]{KubWeakFraisse} in a more general setup.
For another exposition see \cite[Section~2.2]{BBDK}.

\begin{tw} \label{thm:Fclass}
    $\K_\fin$ has a (weak) Fraïssé sequence if and only if it is a (weak) Fraïssé class.
\end{tw}

\begin{tw} \label{thm:Flim}
    Let $\K$ be a $\sigma$-complete category of structures, and let $U \in \sig\K_\fin$.
    The following conditions are equivalent.
    \begin{enumerate}
        \item $U$ is universal and (weakly) homogeneous for $\K_\fin$.
        \item $U$ is universal and has the (weak) extension property for $\K_\fin$.
        \item $U$ is the colimit of a (weak) Fraïssé sequence in $\K_\fin$.
    \end{enumerate}
    Moreover, such object $U$ is unique, and every sequence in $\K_\fin$ with colimit $U$ is (weak) Fraïssé.
    In the Fraïssé case, $U$ is universal for $\sig\K_\fin$.
\end{tw}

The unique object $U$ from Theorem~\ref{thm:Flim} is called the \emph{(weak) Fraïssé limit} of $\K_\fin$ (in $\K$), and by Theorem~\ref{thm:Fclass} it exists if and only if $\K_\fin$ is a (weak) Fraïssé class.

\begin{observation} \label{obs:limit_from_cofinal}
    For any cofinal subcategory $\Su \subseteq \K_\fin$, 
    every weakly absorbing sequence $(A_n, f_n)$ in $\K_\fin$ can be intertwined with a sequence $(B_n, g_n)$ with $B_n \in \Su$ for every $n$.
    From this some properties of the limit may be concluded, as we will see in Corollary \ref{wn:generic-clo or mon}.
\end{observation}


For every structure $X \in \sigma\K_\fin$ we endow its automorphism group $\Aut(X)$ with the topology of pointwise convergence, turning it into a Polish group.
A basic open set is of the form $V_{\seq{A, p}} = \set{f \in \Aut(X) : \seq{X, f} \geq \seq{A, p}}$ where $A \leq X$ and $\seq{A, p} \in \K^*_\fin$.

\begin{df}
Let $\K$ be a $\sigma$-complete category of structures such that $\K_\fin$ is a Fraïssé class, and let $U$ be the Fraïssé limit.
A \emph{generic $n$-tuple of automorphisms} of $U$ is $\seq{h_1,\ldots, h_n}\in \Aut(U)^n$ such that the conjugacy class
\[
    \{\seq{g h_1 g^{-1},\ldots, g h_n g^{-1}}: g \in \Aut(U)\}
\]
is comeager in $\Aut(U)^n$.
\end{df}

It turns out that existence of a generic $n$-tuple of automorphisms is closely connected to Fraïssé-theoretic properties of the class of finite partial automorphisms $\K^{*n}_\fin$.
However, partial automorphisms are essentially incompatible with AP.
By a correspondence of Kechris and Rosendal~\cite[Theorem~6.2]{KR} a generic $n$-tuple exists if and only if the class of finite partial automorphisms satisfies JEP and WAP.
A variant of the correspondence for $\omega$-categorical structures was independently proved by Ivanov~\cite[Theorem~1.2]{Iv}.
We shall use the following more precise formulation, which explicitly mentions  weak Fraïssé limit, in order to deduce properties of a generic automorphism in Corollary~\ref{wn:generic-clo or mon} via Observation~\ref{obs:limit_from_cofinal}.
To our knowledge, this more precise formulation is folklore.

Let us also note that the original theorem is stated for classes of one-sorted structures. This is not a problem as every many-sorted structure can be turned into a one-sorted structure while preserving the notion of embedding, and hence any category of structures is equivalent to a category of one-sorted structures.

\begin{tw} \label{thm:KR}
    Let $\K$  be a category of structures such that $\K_\fin$ is a Fraïssé class and let $U$ be its Fraïssé limit.
    Then for every $n \in \N$ and $\seq{h_1, \ldots, h_n} \in \Aut(U)^n$ the following are equivalent:
    \begin{enumerate}
        \item $\seq{h_1, \ldots, h_n}$ is a generic $n$-tuple of automorphisms of $U$,
        \item $\K_\fin^{*n}$ is a weak Fraïssé class and $\seq{U, h_1, \ldots, h_n} \in \K^{*n}$ is the limit.
    \end{enumerate}
    Hence, a generic $n$-tuple of automorphisms exists if and only if $\K_\fin^{*n}$ has JEP and WAP.
\end{tw}


We shall recall transfer principles formulated in \cite[Section~3.1]{paper1}.
These provide conditions for preservation of amalgamation and of the Fraïssé limit by a functor, and capture the situation of generalized reducts (images under forgetful functors) of a homogeneous structure being homogeneous in their respective categories.

Let us recall the relevant properties of functors.
A functor $F\maps \C \to \D$ is

\begin{itemize}
    \item {\it wide} / {\it surjective on objects} if for every $\D$-object $X$ there is a $\C$-object $X'$ such that $F(X')=X$;
    
    \item {\it essentially wide} / {\it essentially surjective on objects} if for every $\D$-object $X$ there is a $\C$-object $X'$ with $F(X') $ isomorphic to $ X$;
    
	\item {\it cofinal} if for every $\D$-object $X$ there is a $\C$-object $X'$ and a $\D$-map $f\maps X \to F(X')$.
\end{itemize}

\begin{uwgi}
We have the following implications for a functor $F$:

\smallskip
\centerline{wide $\implies$ essentially wide $\implies$ cofinal.}
\end{uwgi}

\noindent
We say that a functor $F\colon \C\to\D$  is
\begin{itemize}
    \item {\it star-surjective} if for every $\C$ object $X'$ and every $\D$-map $f\colon F(X')\to Y$ there is a $\C$-object $Y'$ and a $\C$-map $f'\colon X'\to Y'$ with $F(f')=f$;
    
	\item {\it essentially star-surjective} if for every $\C$-object $X'$ and every $\D$-map $f\maps F(X') \to Y$ there is a $\C$-object $Y'$ and a $\C$-map $f'\maps X' \to Y'$ and an $\D$-isomorphism $g\maps Y \to F(Y')$ with $F(f') = g \circ f$;
    
    \item {\it absorbing} if for every $\C$-object $X'$ and every $\D$-map $f\maps F(X') \to Y$ there is a $\C$-object $Y'$, a $\C$-map $f'\maps X' \to Y'$, and an $\D$-map $g\maps Y \to F(Y')$ with $F(f') = g \cmp f$;  
    
	\item {\it weakly absorbing} if for every $\C$-object $X'$ there is a $\D$-map $e\maps F(X') \to X$ such that for every $\D$-map $f\maps X \to Y$ there is a $\C$-object $Y'$, a $\C$-map $f'\maps X' \to Y'$, and an $\D$-map $g\maps Y \to F(Y')$ with $F(f') = g \circ f \circ e$ (wlog the map $e$ may be of the form $F(e')$ for some $\C$-map $e'\maps X' \to X''$).
\end{itemize}

\begin{uwgi}
We have the following implications for a functor $F$:

\smallskip
\centerline{star-surjective $\Rightarrow$ essentially star-surjective $\Rightarrow$ absorbing $\Rightarrow$ weakly absorbing.}
\end{uwgi}

\begin{uwgi}
    The above properties of functors $\C \to \D$ specialize to the corresponding properties of subcategories $\C \subseteq \D$ (via the inclusion functor) and of sequences in $\D$ (viewed as functors $\N \to \D$).
\end{uwgi}

Note that definitions of AP, CAP, and WAP make sense for arbitrary categories, and we use them in the following general proposition.

\begin{prop} \label{thm:fin_transfer}
    Let $\C$ and $\D$ be categories and $F\maps \C \to \D$ a functor.
    \begin{enumerate}
        \item If $\C$ has AP and $F$ is essentially wide and absorbing, then $\D$ has AP.
        \item If $\C$ has CAP and $F$ is cofinal and absorbing, then $\D$ has CAP.
        \item If $\C$ has WAP and $F$ is cofinal and weakly absorbing, then $\D$ has WAP.
    \end{enumerate}
    In particular, if $F$ is wide and star-surjective, all the amalgamation properties are transferred forward.
\end{prop}

A functor $F\maps \C \to \D$ is {\it $\sigma$-continuous} if  for every sequence $(X_n, f_n)$ in $\C$ and its colimit $(X_\infty, f^\infty_n)$ we have that $(F(X_\infty), F(f^\infty_n))$ is a colimit of $(F(X_n), F(f_n))$ in $\D$.
For categories of structures, where objects are (many-sorted) structures and morphisms are all embeddings, the colimit of a sequence is essentially the union of the structures, and so to check $\sig$-continuity, it is enough to show that for every point $x \in F(X_\infty)$ there is $n \in \omega$ and $x' \in F(X_n)$ such that $f_n^\infty(x') = x$.

\begin{observation}
    If $\C$ is a $\sigma$-complete category of structures, then so is $\C^*$ with the colimit of $(\seq{X_n, p_n}, f_n)$ being $(\seq{X_\infty, p_\infty}, f^\infty_n)$ where $(X_\infty, f^\infty_n)$ is the colimit of $(X_n, f_n)$ and $p_\infty\maps \dom(p_\infty) \to \rng(p_\infty)$ is the colimit-induced isomorphism between the colimit of $(\dom(p_n))$ and the colimit of $(\rng(p_n))$, and similarly for every $\C^{*m}$.
    Hence, if $F\maps \C \to \D$ is a $\sigma$-continuous functor between $\sigma$-complete categories of structures, so is every $F^{*m}\maps \C^{*m} \to \D^{*m}$.
\end{observation}

The first two cases of the following proposition come from \cite[Proposition~3.8]{paper1}. The proof of the third case is analogous and we omit it.

\begin{prop} \label{thm:Flim_transfer}
    Let $\C$ and $\D$ be $\sigma$-complete categories of structures, and let $F\maps \C \to \D$ be a $\sigma$-continuous functor that restricts to $F_\fin\maps \C_\fin \to \D_\fin$.
    \begin{enumerate}
        \item If $\C_\fin$ is Fraïssé with limit $U$, and $F_\fin$ is essentially wide and absorbing, then $\D_\fin$ is Fraïssé with limit $F(U)$.
        
        \item If $\C_\fin$ is Fraïssé with limit $U$, $F_\fin$ is cofinal and absorbing, and $\D_\fin$ has AP, then $\D_\fin$ is Fraïssé with limit $F(U)$.
        
        \item If $\C_\fin$ is weak Fraïssé with limit $U$, and $F_\fin$ is cofinal and weakly absorbing, then $\D_\fin$ is weak Fraïssé with limit $F(U)$.
    \end{enumerate}
\end{prop}

By applying Proposition~\ref{thm:Flim_transfer}(3) to $F^{*n}$ and by using Theorem~\ref{thm:KR} we immediately obtain the following.

\begin{wn} \label{thm:generic_transfer}
    Let $\C$ and $\D$ be $\sigma$-complete categories of structures, and let $F\maps \C \to \D$ be a $\sigma$-continuous functor that restricts to $F_\fin\maps \C_\fin \to \D_\fin$.
    Suppose that $\C_\fin$ is Fraïssé with limit $U$ and that $\D_\fin$ is Fraïssé with limit $F(U)$.
    
    If $\seq{h_1, \ldots, h_n} \in \Aut(U)^n$ is a generic $n$-tuple, and $F^{*n}_\fin$ is cofinal and weakly absorbing, then $\seq{F(h_1), \ldots, F(h_n)} \in \Aut(F(U))^n$ is a generic $n$-tuple.
\end{wn}

\subsection{Two-sorted ultrametric spaces}\label{prel:2sort}

By $\Lin$ we denote the category of all linear orders $\seq{L, \leq_L}$ and all embeddings, i.e. maps $f\maps L \to L'$ such that $x \leq_L y$ if and only if $f(x) \leq_{L'} f(y)$.
It is well-known that $\Lin_\fin$ is a Fraïssé class, whose limit is the linearly ordered set of rationals $\Q$.

We shall consider ultrametric spaces with linear orders as sets of distances.
First, let $\Lin^0$ denote the category of linear orders $\seq{L, \leq_L, 0_L}$ with the least element $0_L$ and of embeddings of linear orders preserving the least element.
As usual, the triple $\seq{L, \leq_L, 0_L}$ is often denoted just by $L$, and $\leq_L$ and $0_L$ are often denoted just by $\leq$ and $0$, respectively.
Clearly, $\Lin^0$ is a category equivalent to $\Lin$.

In our previous paper~\cite{paper1}, we studied ultrametric spaces as two-sorted structures of the form $\seq{X, d_X, D_X}$,
where $X$ is a set of points, $D_X$ is a set of distances, i.e. an $\Lin^0$-object, and $\map{d_X}{X \times X}{D_X}$ is an ultrametric.
To be more specific, $d = d_X$ is symmetric, $d(x, x') = 0$ if and only if $x = x'$, and $d$ satisfies the ultrametric triangle inequality 
\[
    d(x, x'') \leq \max\set{d(x,x'), d(x',x'')} \qquad \text{for $x,x',x'' \in X$.}
\]
Again, we shall denote $d_X$ by $d$ and $\seq{X, d_X, D_X}$ by $\seq{X, D_X}$ or just $X$ whenever convenient.
We shall refer to the two sorts as the \emph{point sort} and the \emph{distance sort}, respectively, when convenient.

A \emph{distance-carrying embedding} $f\maps \seq{X, d_X, D_X} \to \seq{Y, d_Y, D_Y}$ (\emph{dc-embedding} for short) is a one-to-one map $f\maps X \to Y$ together with an $\Lin^0$-map $D_f\maps D_X \to D_Y$ such that
\[
    d_Y(f(x), f(x')) = D_f(d_X(x, x')) \qquad \text{for } x, x' \in X.
\]
Note that an \emph{isometric embedding} corresponds to a dc-embedding $f\maps X \to Y$ such that $D_f$ is the inclusion $D_X \leq D_Y$, i.e. we have $d_Y(f(x), f(x')) = d_X(x, x')$.

Let $\Ult$ denote the category of ultrametric spaces and distance-carrying embeddings.
Clearly the assignment $(f\maps X \to Y) \mapsto (D_f\maps D_X \to D_Y)$ yields a functor $D\maps \Ult \to \Lin^0$ with the restriction $D_\fin \maps \Ult_\fin \to \Lin^0_\fin$ (which we may denote just by $D$).
Note that a $\Ult_\fin$-object is an ultrametric space $X$ such that both the set of points $X$ and the set of distances $D_X$ are finite.

Let $\UltIso^E \subs \Ult$ for $E \in \Lin^0$ denote the subcategory of all ultrametric spaces $X$ with $D_X \leq E$ and all isometric embeddings.
We say that an ultrametric space $X$ is \emph{precise} if $d_X\maps X \times X \to D_X$ is surjective, i.e. every distance is achieved.
By $\UltIso^E_{\rm prec} \subs \UltIso^E$
  we mean the full subcategory consisting of precise spaces. 
In this view, the category of classical ultrametric spaces is just $\UltIso^{\Rpos}_{\rm prec}$.
A {\it precise partial dc-automorphism} $p$ of $A$ is a  partial automorphism such that all $\dom(p), \rng(p) \leq A$ are precise.
Then for every $r \in \dom(D_p)$ we have $D_p(r) = D_p(d(x, y)) = d(p(x), p(y))$ for some $x, y \in A$, i.e.  we can ignore the distance part of $D_p$.

Since every two-sorted ultrametric space $X$ can be viewed both as a member of $\Ult$ and $\UltIso^E$ (for any $E \in \Lin^0$ with $D_X \leq E$), we make a convention that $\Aut(X) = \Aut_{\Ult}(X)$ denotes the dc-automorphism group, while $\Iso(X) = \Aut_{\UltIso^E}(X)$ denotes the isometry group, which does not depend on $E$.
We say that $X$ is \emph{dc-homogeneous} if it is homogeneous for $\Ult_\fin$, and that $X$ is \emph{iso-homogeneous} if it is homogeneous for $\UltIso^E_\fin$.

Recall that the class of classical finite rational ultrametric spaces with isometric embeddings, which in our setup is isomorphic to the category $(\UltIso^{\Qpos}_{\rm prec})_\fin$, is Fraïssé and that its Fraïssé limit is the \emph{countable rational Urysohn ultrametric space} $\U_\Q$.
In the previous paper~\cite{paper1}, we showed that finite two-sorted ultrametric spaces also  form a Fraïssé class, with the same limit.

\begin{tw}[{\cite[Theorem~3.11]{paper1}}] \label{thm:U_Fraisse}
    $\Ult_\fin$ is a Fraïssé class, and its Fraïssé limit $\U$ is dc-isomorphic to $\U_\Q$.
    In particular, $\U$ is both dc-homogeneous and iso-homogeneous.
\end{tw}

\section{Description of a single orbit}

In this section, we give a complete description of types of single orbits of automorphisms of two-sorted ultrametric spaces.

Let $\seq{X, f}$ be a total dc-automorphism.
We shall analyze the orbit of a point $a \in X$, which consists of the points $a_n = f^n(a)$ for $a \in \Z$.
First let us denote the sequence of the forward distances from $a$ by $(r_n)$, i.e. $r_n = d(a_0, a_n)$ for every $n \geq 1$.

Let $L \geq 1$ be the smallest integer such that $r_L = 0$, if it exists; otherwise put $L = \infty$.
We call $L$ the \emph{length} of the orbit.
Clearly, $L < \infty$ means that $f^\Z(a)$ is a \emph{closed} orbit of length $L$ as $a_L = a_0$.

Let $N \geq 1$ be the smallest integer such that $D_f(r_N) \neq r_N$, if it exists; otherwise put $N = \infty$.
Note that for every $n < N$ the distance $r_n$ is a fixed point of $D_f$, and so $d(a_k, a_{k + n}) = r_n$ for every $k \in \Z$.
On the other hand, $r_N$ induces the monotone orbit of $D_f$ consisting of the distances $(D_f)^k(r_N) = d(a_k, a_{k + N})$ for $k \in \Z$.
Hence, we call the orbit \emph{monotone} if $N < \infty$, and we call $N$ the \emph{monotone period} of the orbit.
A monotone orbit is either \emph{decreasing}, i.e. when $D_f(r_N) < r_N$, or \emph{increasing}, i.e. when $D_f(r_N) > r_N$.
We also call $r_n$ a \emph{fixed distance} or a \emph{monotone distance} if $D_f(r_n) = r_n$ or $D_f(r_n) \neq r_n$, respectively.

The cases of being closed or monotone are mutually exclusive since a closed orbit attains finitely many distances, while a monotone orbit attains infinitely many.
In the third case, when $L = N = \infty$, we call the orbit \emph{horizontal}.
Horizontal orbits are of two subtypes depending on whether the set $\set{r_n: n \geq 1}$ of all attained distances is finite or infinite.

\begin{observation} \label{obs:invariant}
    The length $L$, the monotone period $N$, the set $\set{r_n: 1 \leq n < N}$ of fixed distances, and so the classification of the orbit as being closed, monotone, or horizontal (with finitely or infinitely many distances) depends only on the orbit $f^\Z(a)$ as a whole, not on the particular choice of the point $a$.
    This is because $r_n = d(a_k, a_{k + n})$ for every $k \in \Z$ whenever $D_f(r_n) = r_n$ (which includes $r_n = 0$).

    Moreover, the invariants do not depend on the orientation of $f$, i.e. if we replace $f$ by $f^{-1}$, we obtain the same $L$, $N$, and $\set{r_n: 1 \leq n < N}$. Only an increasing monotone orbit becomes decreasing monotone and vice versa.
    This is because if $(a'_n)$ and $(r'_n)$ denote the points and distances with respect to $f^{-1}$, we have $r'_n = d(a'_0, a'_n) = d(a_0, a_{-n}) = (D_f)^{-n}(d(a_n, a_0)) = (D_f)^{-n}(r_n) = r_n$ whenever $D_f(r_n) = r_n$.
\end{observation}

Let us adopt a convention that $n \leq N$ means $n \in \set{1, 2, \ldots, N}$ if $N < \infty$, and $n \in \set{1, 2, \ldots}$ if $N = \infty$.

\begin{lm} \label{thm:decreasing_distances}
    If $r_n \notin \set{r_m: m < n}$ for some $1 < n \leq N$, then $r_n < \min\set{r_m: m < n}$.\kern-0.4em
\end{lm}
\begin{proof}
    Let $1 \leq m < n$.
    Note that $d(a_m, a_n) = d(a_0, a_{n - m})$ since $n - m < N$.
    We consider the triangle with vertices $a_0, a_m, a_n$, which attains the distances $r_m, r_{n - m}, r_n$.
    Since, $r_n \neq r_m, r_{n - m}$, by the ultrametric triangle inequality we have $r_n < r_m = r_{n - m}$.
\end{proof}

Let $(n_i)$ be the increasing enumeration of the set $\set{n \leq N: r_n \notin \set{r_m: m < n}}$ of indices where a new distance (possibly up to $r_N$) appears.
So $n_0 = 1$, and we either obtain a finite sequence $(n_i)_{0 \leq i \leq k}$, or an infinite sequence $(n_i)_{0 \leq i < \infty}$.
It is easy to check that 
for a closed orbit of length $L$, $(n_i)$ is finite with $n_k = L$,
for a monotone orbit of period $N$, $(n_i)$ is finite with $n_k = N$, and
for a horizontal orbit, $(n_i)$ can be either finite, or infinite, depending on the set of attained distances.
We call $(n_i)$ the \emph{constructing sequence} for the orbit $f^\Z(a)$.
Again, it depends only on the orbit as a whole and agrees with the constructing sequence for $f^{-1}$ by Observation~\ref{obs:invariant}.

\begin{lm} \label{thm:constructing_sequence}
    Every element $n_{i + 1}$ of the constructing sequence is a non-trivial positive multiple of $n_i$, i.e. $n_{i + 1} = l_i n_i$ for some $l_i \in \N \setminus \set{0, 1}$.
\end{lm}
\begin{proof}
    By definition we have $n_{i + 1} > n_i$, and so $n_{i + 1}$ can be uniquely written as $l_i n_i + j$ for $l_i \in \N \setminus \set{0}$ and $j \in \set{0, \ldots, n_i - 1}$.
    Note that $d(a_0, a_{n_i}) = d(a_{n_i}, a_{2n_i}) = \cdots = d(a_{(l_i - 1) n_i}, a_{l_i n_i}) = r_{n_i}$, and so by the ultrametric triangle inequality we have $r_{l_i n_i} \leq r_{n_i}$.
    If $j > 0$, then the triangle with the vertices $a_0, a_{l_i n_i}, a_{n_{i + 1}}$ attains the distances $r_{l_i n_i}, r_{n_{i + 1}}, r_j$, and we have have $r_{n_{i + 1}} < r_{l_i n_i} \leq r_{n_i} < r_j$, which is a contradiction.
    Hence, $j = 0$ and $l_i \neq 1$ since $n_{i + 1} \neq n_i$.
\end{proof}

\begin{lm} \label{thm:horizontal_distances}
     Let $m < n \in \Z$.  Suppose that $N=\infty$ or $N<\infty$ and $N$ does not divide $n - m$. Then  $d(a_m, a_n) = r_{n_i}$ where $n_i$ is the largest element of the constructing sequence that divides $n - m$.
\end{lm}
\begin{proof}
    It is enough to prove the claim in the case $m = 0$ since then $d(a_m, a_n) = d(f^m(a_0), f^m(a_{n - m})) = (D_f)^m(d(a_0, a_{n - m})) = (D_f)^m(r_{n_i}) = r_{n_i}$, where the last equality holds since we have $n_i < N = n_k$ by the assumption if $N < \infty$.
    Suppose that $m = 0$, and so $n_i$ is the largest element of the constructing sequence that divides $n - m = n$, i.e. $n = l n_i$ for some $l \geq 1$.
    We want to show $r_n = r_{n_i}$.
    As in the proof of Lemma~\ref{thm:constructing_sequence}, we have $r_n = r_{l n_i} \leq r_{n_i}$.
    If $n_i$ is the last element of the constructing sequence, we have $r_n \geq r_{n_i}$ and we are done.
    If $n \leq n_{i + 1}$ (and by the assumption necessarily $n < n_{i + 1}$), we also have $r_n = r_{n_i}$ as $n_{i + 1}$ is the first time the distance drops below $r_{n_i}$.
    Finally, if $n > n_{i + 1}$, we can write $n = l' n_{i + 1} + n'$ for $l' \geq 1$ and $1 \leq n' < n_{i + 1}$.
    Then the triangle with the vertices $a_0, a_{l' n_{i + 1}}, a_n$ attains the distances $r_{l' n_{i + 1}}, r_n, r_{n'}$, and we have $r_{l' n_{i + 1}} \leq r_{n_{i + 1}} < r_{n'} = r_n$.
    But since both $n$ and $n_{i + 1}$ are divisible by $n_i$ (the latter by Lemma~\ref{thm:constructing_sequence}), we have $r_{n'} = r_{l'' n_i} = r_{n_i}$, for some $l''\geq 1$, by the previous case for $n'$ instead of $n$.
\end{proof}

The previous lemma describes the ultrametric for all pairs $a_m \neq a_n$ of a closed or horizontal orbit and for some pairs of a monotone orbit, with the values being all fixed distances of the orbit.
Unless the orbit is horizontal with infinitely many distances, we have a finite constructing sequence $(n_i)_{0 \leq i \leq k}$.
Then $r_{n_k} = 0$ if the orbit is closed, $r_{n_k}$ is the smallest positive distance if the orbit is horizontal, and $r_{n_k} = r_N$ is a monotone distance if the orbit is monotone.
Hence, for a closed or monotone orbit, the set of all positive fixed distances is $\set{u_1, \ldots, u_k} := \set{r_{n_0}, \ldots, r_{n_{k - 1}}}$, while for a horizontal orbit, the set of all positive fixed distances is $\set{u_1, \ldots, u_{k + 1}} := \set{r_{n_0}, \ldots, r_{n_k}}$ in the finite case and $\set{u_1, u_2, \ldots} := \set{r_{n_0}, r_{n_1}, \ldots}$ in the infinite case.

For monotone orbits we need to consider the monotone distances.
It turns out to be convenient to associate the monotone distances with the points via the map $\rho\maps f^\Z(a) \to D_X$ defined by
\[
    \rho(a_n) = \min\set{d(a_{n - N}, a_n), d(a_n, a_{n + N})}.
\]
Then $r_N$ becomes either $\rho(a_0)$ or $\rho(a_N)$, depending on whether the monotone orbit is decreasing or increasing, but we have $D_f(\rho(a_n)) = \rho(a_{n + 1})$ in both cases, and both $f$ and $f' = f^{-1}$ give the same map $\rho$: it is easy to check that $\rho'(a'_n) = \rho(a_{-n})$ and so $\rho' = \rho$ as $a'_n = a_{-n}$ for any $n \in \Z$.

If we put $v_n = \rho(a_n)$ for a decreasing monotone orbit and $v_n = \rho(a_{-n})$ for an increasing monotone orbit, $\set{v_n}_{n \in \Z}$ becomes the set of all monotone distances enumerated in the decreasing order.

\begin{lm} \label{thm:monotone_distances}
    In a monotone orbit we have $d(a_m, a_n) = \max\set{\rho(a_m), \rho(a_n)}$ for every $m < n \in \Z$ such that $N$ divides $n - m$.
\end{lm}
\begin{proof}
    It is enough to consider the case $m = 0$ since in general we would obtain 
    \[
        d(a_m, a_n) = (D_f)^m(d(a_0, a_{n - m})) = (D_f)^m(\max\set{\rho(a_0), \rho(a_{n - m})}) = \max\set{\rho(a_m), \rho(a_n)}.
    \]
    So suppose $m = 0$ and that $n$ is divisible by $N$ and $l$ is such that $n=lN$.
    If $D_f(r_N) < r_N$, we get $d(a_{-N},a_0)=(D_f)^{-N}(r_N)<r_N$. Therefore we obtain $\rho(a_0) = r_N$ and $\rho(a_{lN}) = (D_f)^{l N}(r_N) = r_{(l + 1) N} < r_N$, and hence 
    \[
        d(a_0, a_n) = d(a_0, a_{l N}) = r_N = \max\set{r_N, r_{(l + 1) N}} = \max\set{\rho(a_0), \rho(a_{l N})}.
    \]
    If $D_f(r_N) > r_N$, we consider $f' = f^{-1}$ and obtain
    \[
        d(a_m, a_n) = d(a'_{-n}, a'_{-m}) = \max\set{\rho'(a'_{-n}), \rho'(a'_{-m})} = \max\set{\rho(a_m), \rho(a_n)}.
        \qedhere
    \]
\end{proof}

The previous lemmas already give a complete description what a single orbit has to look like up to isomorphism.
To give an explicit enumeration of the isomorphism types we consider the following concrete models of orbits.

\medskip

We start with an \emph{(abstract) constructing sequence}, i.e. a finite sequence $(n_i)_{0 \leq i \leq k}$ for $k \in \N$ or an infinite sequence $(n_i)_{0 \leq i < \infty}$ such that $n_0 = 1$ and $n_{i + 1} = l_i n_i$ for some $l_i \in \N \setminus \set{0, 1}$ for every $i$.

Let $\Z$ denote the infinite cyclic group of integers, for every $n \in \N \setminus \set{0, 1}$ let $\Z_n$ denote the finite cyclic group of order $n$, and let $h_n\maps \Z \to \Z_n$ denote the canonical homomorphism mapping $1 \mapsto 1$.
For every constructing sequence $(n_i)$ we let $\Z_{(n_i)}$ be the cyclic subgroup of $\prod_i \Z_{n_i}$ generated by the constant one sequence $\bar{1}$.
We have the canonical surjective homomorphism $h_{(n_i)}\maps \Z \to \Z_{(n_i)}$ mapping every $n \in \Z$ to the constant sequence $\bar{n}$, i.e. $h_{(n_i)}$ is the diagonal product of the homomorphisms $h_{n_i}$.
If $(n_i)_{0 \leq i \leq k}$ is finite, $h_{(n_i)}$ induces an isomorphism $\Z_{n_k} \to \Z_{(n_i)}$, while if $(n_i)$ is infinite, $h_{(n_i)}$ is already an isomorphism.
For a finite constructing sequence $(n_i)$ we also let $\Z_{(n_i), \infty}$ be the cyclic subgroup of $(\prod_i \Z_{n_i}) \times \Z$ generated by $\bar{1}$ and we let $h_{(n_i), \infty}\maps \Z \to \Z_{(n_i), \infty}$ denote the corresponding canonical isomorphism.

In all the cases considered, i.e. $\Z_{(n_i)}$ and $\Z_{(n_i), \infty}$ for $(n_i)$ finite, and $\Z_{(n_i)}$ for $(n_i)$ infinite, we let $(a_n)_{n \in \Z}$ be the canonical enumeration of $\Z_{(n_i)}$ (or $\Z_{(n_i), \infty}$), i.e. $a_n = h_{(n_i)}(n) = \bar{n}$ (or $a_n = h_{(n_i),\infty}(n)$).
We also let $s\maps \Z_{(n_i)} \to \Z_{(n_i)}$ (or $s\maps\Z_{(n_i), \infty} \to \Z_{(n_i), \infty}$) denote the successor operation $a_n \mapsto a_{n + 1}$, $n \in \Z$, which is an automorphism of the set $\Z_{(n_i)}$ (or $\Z_{(n_i), \infty}$) with a single orbit.
We will turn the sets $\Z_{(n_i)}$ and $\Z_{(n_i), \infty}$ into ultrametric spaces such that $s$ becomes a dc-automorphism.

Let us fix a linearly ordered set of distances
\[
    E = \set{u_1 > u_2 > \cdots > 0} = \set{u_i}_{i \in \N}  \cup \set{0}
\]
and its finite parts $E_k = \set{u_i}_{1 \leq i \leq k} \cup \set{0} \subs E$ for $k \in \N$.
Let $(n_i)_{i \in I}$ be a constructing sequence, so $I = \set{0, \ldots, k}$ or $I = \N$, and for every $j \in I$ let $\pi_j\maps \Z_{(n_i)} \to \Z_{n_j}$ denote the projection.
We endow $\Z_{(n_i)}$ with the following ultrametric:
\begin{align*}
    &d(a_m, a_n) = u_j & \text{where } j &= \min\set{i \in I: \pi_i(a_m) \neq \pi_i(a_n)} \\
        && &= \min\set{i \in I: n_i \text{ does not divide } n - m} \geq 1, & \text{for $a_m \neq a_n$}.
\end{align*}
That turns $\Z_{(n_i)}$ into a two-sorted ultrametric space with the distance set $E_k$ (in the finite case) or $E$ (in the infinite case), and $s\maps a_n \mapsto a_{n + 1}$ to  an isometry.

For a finite constructing sequence $(n_i)_{0 \leq i \leq k}$ we define two ultrametric structures on $\Z_{(n_i),\infty}$.
The distance $d$ is defined by the same formula as above, but we also consider the last projection $\pi_{k + 1} \maps \Z_{(n_i), \infty} \to \Z$, so the set of distances is $E_{k + 1}$.
Again, $s$ becomes an isometry of $\Z_{(n_i), \infty}$, fixing all distances.

To define the second ultrametric we first consider additional distances $\set{v_n}_{n \in \Z}$ such that we obtain
\[
    E^* = \set{u_1 > u_2 > \cdots > v_{-1} > v_0 > v_1 > \cdots > 0} = E \cup \set{v_n}_{n \in \Z}.
\]
Let $\phi \in \aut(E^*)$ be such that $\phi(0) = 0$, $\phi(u_i) = u_i$ for every $i$, and $\phi(v_n) = v_{n + 1}$ for every $n$.
For every $k \in \N$ let $E^*_k$ be the part $\set{u_i}_{1 \leq i \leq k} \cup \set{v_n}_{n \in \Z} \cup \set{0} \subs E^*$ and let $\phi_k = \phi\restr{E^*_k} \in \aut(E^*_k)$.
The second ultrametric on $\Z_{(n_i), \infty}$ is defined by 
\[
d^*(a_m, a_n) = \begin{cases}
    u_j & \text{if $j := \min\set{i \leq k + 1: \pi_i(a_m) \neq \pi_i(a_n)} < k + 1$,} \\
    \max\set{v_m, v_n} & \text{if $j = k + 1$, or equivalently $n_k$ divides $n - m$.}
\end{cases}
\]
The distance set is $E^*_k$, and $s$ becomes dc-automorphism with the distance part $D_s = \phi_k$.

\begin{observation}
    To show the ultrametric triangle inequality for $d^*$ we pick distinct points $a_l, a_m, a_n$ with $l, m, n \in \Z$.
    We know that $d$ is an ultrametric on $\Z_{(n_i), \infty}$.
    Note that $d^*$ modifies $d$ in a way that when $d(a_m, a_n)$ takes the smallest positive value $u_{k + 1}$, $d^*(a_m, a_n)$ takes  one of the values $v_l$ instead,  which is even smaller.
    So the ultrametric triangle inequality for $d^*$ is satisfied when zero or one of the distances $d(a_l, a_m), d(a_m, a_n), d(a_l, a_n)$ is equal to $u_{k + 1}$.
    It is impossible for exactly two of the distances to be $u_{k + 1}$, and in the case that all three distances are $u_{k + 1}$ it is easy to see that the assignment $\seq{a_m, a_n} \mapsto \max\set{v_m, v_n}$ itself satisfies ultrametric triangle inequality.
\end{observation}

We are ready to define particular total automorphisms of two-sorted ultrametric spaces with a single orbit that represent all possible isomorphism types, as we shall prove.
\begin{itemize}
    \item $C_{(n_i)} = \seq{\seq{\Z_{(n_i)}, d, E_k}, \seq{s, \id{}}}$ for a finite constructing sequence $(n_i)_{0 \leq i \leq k}$ is a closed orbit of length $n_k > 0$.
    This includes the case of a fixed-point orbit ($k = 0$).
    
    \item $H_{(n_i)} = \seq{\seq{\Z_{(n_i)}, d, E}, \seq{s, \id{}}}$ for an infinite constructing sequence $(n_i)_{0 \leq i < \infty}$ is a horizontal orbit with infinitely many distances.

    \item $H_{(n_i), \infty} = \seq{\seq{\Z_{(n_i), \infty}, d, E_{k + 1}}, \seq{s, \id{}}}$ for a finite constructing sequence $(n_i)_{0 \leq i \leq k}$ is a horizontal orbit with finitely many distances.

    \item $M^-_{(n_i), \infty} = \seq{\seq{\Z_{(n_i), \infty}, d^*, E^*_k}, \seq{s, \phi_k}}$ 
    and $M^+_{(n_i), \infty} = \seq{\seq{\Z^*_{(n_i), \infty}, d^*, E^*_k}, \seq{s^{-1}, \phi_k^{-1}}}$ 
    for a finite constructing sequence $(n_i)_{0 \leq i \leq k}$ are a decreasing and an increasing monotone orbit.
\end{itemize}

\begin{tw} \label{thm:single_orbit}
    Let $\seq{X, f} \in \Ult^*$ be a total dc-automorphism with a single orbit of the point sort and all distances used.
    Then $\seq{X, f}$ is isomorphic to exactly to one of the following:
    \begin{enumerate}
        \item $C_{(n_i)} $ for a finite constructing sequence $(n_i)_{0 \leq i \leq k}$,
        \item $H_{(n_i)} $ for an infinite constructing sequence $(n_i)_{0 \leq i < \infty}$,
        \item $H_{(n_i), \infty} $ for a finite constructing sequence $(n_i)_{0 \leq i \leq k}$,
        \item $M^-_{(n_i),\infty} $ or $M^+_{(n_i),\infty}$ for a finite constructing sequence $(n_i)_{0 \leq i \leq k}$.
    \end{enumerate}
\end{tw}
\begin{proof}
    We pick any $a \in X$ and define $(a_n)$, $(r_n)$, $L$, $N$, and the constructing sequence $(n_i)$ as above.
    If $N = \infty$, then Lemma~\ref{thm:horizontal_distances} shows that $d_X$ agrees with the ultrametric $d$ of $C_{(n_i)}$, $H_{(n_i)}$, and $H_{(n_i), \infty}$.
    Note that $u_i = r_{n_{i + 1}}$ and the smallest $i$ such that $n_{i + 1}$ does not divide $n - m$ is the largest $i$ such that $n_i$ divides $n - m$.
    Hence, the enumeration $(a_n)$ gives the isomorphism between $\seq{X, f}$ and $C_{(n_i)}$, $H_{(n_i)}$, or $H_{(n_i), \infty}$, depending on type of the orbit.

    Similarly, if $N < \infty$, i.e. if the orbit is monotone, we also define $\rho\maps X \to D_X$ as above, and by combining Lemma~\ref{thm:horizontal_distances} with Lemma~\ref{thm:monotone_distances} we get that $d_X$ agrees with the ultrametric $d^*$ of $M^-_{(n_i), \infty}$ and $M^+_{(n_i), \infty}$.
    Note that $N = n_k$ divides $n - m$ if and only if $\pi_i(a_m) \neq \pi_i(a_n)$ just for $i = k + 1$.
    For a decreasing monotone orbit we have $\rho(a_n) = v_n$ and $D_f = \phi$, and so $\seq{X, f} \cong M^-_{(n_i), \infty}$.
    For an increasing monotone orbit we have $\seq{X, f^{-1}} \cong M^-_{(n_i), \infty}$, and so $\seq{X, f} \cong M^+_{(n_i), \infty}$.
\end{proof}

In the following, $\seq{X, f}\restr{L}$ for $L \in \N$ and for $\seq{X, f}$ being any of the orbit type spaces $C_{(n_i)}$, $H_{(n_i)}$, $H_{(n_i),\infty}$, $M^-_{(n_i),\infty}$, and $M^+_{(n_i),\infty}$, denotes the restriction to the subset $\set{a_n: 0 \leq n < L}$ of the first $L$ points (and to distances used by those points).
We define the \emph{partial horizontal} and \emph{partial monotone} orbits:
\begin{itemize}
    \item $H_{(n_i), L} = H_{(n_i), \infty}\restr{L}$ for a finite constructing sequence $(n_i)_{0 \leq i \leq k}$ and $L \geq n_k + 1$,
    \item $M^-_{(n_i), L} = M^-_{(n_i), \infty}\restr{L}$ and $M^+_{(n_i), L} = M^+_{(n_i), \infty}\restr{L}$ for a finite constructing sequence $(n_i)_{0 \leq i \leq k}$ and $L \geq n_k + 2$.
\end{itemize}

\begin{observation} \label{obs:orbit_restrictions_isomorphic}
    Note that $H_{(n_i), L}$ is (uniquely) isomorphic both to $H_{(n'_i), \infty}\restr{L}$ and to $H_{(n'_i)}\restr{L}$ for every finite constructing sequence $(n'_i)_{0\leq i \leq k'}$ with $k' > k$ and every infinite constructing sequence $(n'_i)_{0\leq i < \infty}$, respectively, such that $n'_i = n_i$ for $i \leq k$, and for $L \leq n'_{k + 1}$.
    This is because for $0 \leq m < n < n'_{k + 1}$ the formulas for $d(a_m, a_n) \in \set{u_1, \ldots, u_{k + 1}}$ agree in $H_{(n_i), \infty}$, $H_{(n'_i)_{i\leq k'}, \infty}$, and $H_{(n'_i)_{i<\infty}}$, whereas $d(a_0, a_{n'_{k + 1}})$ is $u_{k + 1}$ in $H_{(n_i), \infty}$ but $u_{k + 2}$ in $H_{(n'_i)_{i\leq k'}, \infty}$ and $H_{(n'_i)_{i < \infty}}$.

    Also note that the only difference between $C_{(n_i)_{i\leq k}}$ and $H_{(n_i)_{i < k}, n_k}$ with $k \geq 1$ is whether the partial automorphism maps $a_{n_k - 1} \mapsto a_0$ or is undefined at $a_{n_k - 1}$.
    The distance formulas as well as the distance parts of the automorphisms agree.

    Finally note that $M^-_{(n_i),\infty}\restr{n_k + 1}$ and $M^+_{(n_i),\infty}\restr{n_k + 1}$ are uniquely isomorphic, with the smallest positive distance $v_0$ at which the automorphism is undefined, while $H_{(n_i), n_k + 1}$ is uniquely isomorphic to the above with $v_0$ replaced by $u_{k + 1}$ that is fixed by the automorphism.
    
\end{observation}

Motivated by the previous observation, we define the \emph{undecided} partial orbits:
\begin{itemize}
    \item 
    $P_{(n_i), L}=
    \begin{cases}
    M^-_{(n_i), \infty}\restr{L} & \text{ if } L = n_k + 1 \\  H_{(n_i), L} & \text{ if } L > n_k + 1,
    \end{cases}$
    
    where $(n_i)_{0 \leq i \leq k}$
    is a finite constructing sequence and $L \geq n_k + 1$, and

    \item $P_{(), 1} = M^-_{(1), \infty}\restr{1}$ as a special notation for an empty automorphism of a one-point space.
\end{itemize}
An orbit type $P_{(n_i), L}$ is called \emph{full} if $L$ is a multiple of $n_k$ (including $P_{(), 1})$, and is called \emph{fresh} if $L = n_k + 1$ (excluding $P_{(), 1}$; only $P_{(1), 2}$ is simultaneously full and fresh).
The point is that when we consider forward one-point orbit extensions, every $P_{(n_0, \ldots, n_k), L}$ can ``continue'' to $P_{(n_0, \ldots, n_k), L + 1}$, but a full one can also be ``closed'' to $C_{(n_0, \ldots, n_k, L)}$ or can ``ascend'' to $P_{(n_0, \ldots, n_k, L), L + 1}$,  which is fresh.
(Note that the closing is actually not a one-point extension – the partial automorphism is extended without adding a new point.)
A fresh orbit type $P_{(n_0, \ldots, n_k), n_k + 1}$ can besides continuing also be extended to $M^-_{(n_0, \ldots, n_k), n_k + 2}$ or $M^+_{(n_0, \ldots, n_k), n_k + 2}$ by making the fresh distance $r_{n_k}$ monotone (putting $d(a_1, a_{n_k + 1}) \neq d(a_0, a_{n_k})$). 
A partial monotone orbit $M^{-/+}_{(n_i), L}$ can only continue  to $M^{-/+}_{(n_i), \infty}$, while an undecided partial orbit $P_{(n_i), L}$, if it eventually does not become closed or monotone, ends up being a horizontal orbit with finitely or infinitely many distances, depending on whether we ascend finitely or infinitely many times.

\begin{wn} \label{thm:partial_single_orbit}
    Let $\seq{A, p} \in \Ult^*_\fin$ be a partial dc-automorphism with a single orbit of the point sort and all distances used.
    Then $\seq{A, p}$ is isomorphic to exactly one of the following:
    \begin{enumerate}
        \item $C_{(n_i)}$ for a finite constructing sequence $(n_i)_{0 \leq i \leq k}$,
        \item $P_{(n_i), L}$ for a finite constructing sequence $(n_i)_{0 \leq i \leq k}$ and $L \geq n_k + 1$, or $P_{(), 1}$,
        \item $H_{(n_i), L}$ for a finite constructing sequence $(n_i)_{0 \leq i \leq k}$ and $L = n_k + 1$,
        \item $M^-_{(n_i), L}$ or $M^+_{(n_i), L}$ for a finite constructing sequence $(n_i)_{0 \leq i \leq k}$ and $L \geq n_k + 2$.
    \end{enumerate}
\end{wn}

\begin{proof}
    
    If $\seq{A, p}$ is a closed orbit, then by Theorem~\ref{thm:single_orbit} it isomorphic to exactly one of the $C_{(n_i)}$'s.
    Otherwise, $A$ can be uniquely enumerated as $a_0 \mapsto a_1 \mapsto \cdots \mapsto a_{L - 1}$ for some $L \geq 1$, and we have have $\dom(D_p) = D_A$ or $\dom(D_p) = D_A \setminus \set{d(a_0, a_{L - 1})}$.
    Since every finite partial automorphism can be extended to a total one and since total single orbits were characterized by Theorem~\ref{thm:single_orbit}, $\seq{A, p}$ is uniquely isomorphic to the restriction $\seq{X, f}\restr{L}$ or to its modification  obtained by undefining  $D_p(d(a_0, a_{L - 1}))$, where $\seq{X, f}$ is one of the single orbit types.
    It follows from Observation~\ref{obs:orbit_restrictions_isomorphic} that $\seq{A, p}$ is isomorphic to one of $P_{(n_i), L}$, $H_{(n_i), L}$, $M^{-/+}_{(n_i), L}$: 
    for $L < n_k$ we have $C_{(n_i)_{i\leq k}}\restr{L} \cong H_{(n_i)_{i < k}, n_k}\restr{L} = H_{(n_i)_{i < k}, \infty}\restr{L} \cong H_{(n_i)_{i \leq k'}, L}$ where $k' < k$ is such that $n_{k'} < L \leq n_{k' + 1}$, and similarly in other cases.

    It remains to show that all the listed types of finite partial orbits are different.
    As already observed, $C_{(n_i)}$'s are all different and different from the others.
    The other are differentiated by the cardinality of the point sort $L$, the (partial) constructing sequence $(n_i)$, and by their signature behavior:
    $M^{-/+}_{(n_i), L}$ has a monotone decreasing/increasing distance since $L \geq n_k + 2$, while $P_{(n_i), L}$ and $H_{(n_i), L}$ have all distances fixed – with the exception of the smallest positive distance of $P_{(n_i), L}$, which is not in $\dom(D_p)$.
\end{proof}

\color{black}

\medskip
By the previous observations,  an isolated orbit can be extended to a closed or monotone one.
The following theorem shows that this remains true in presence of other orbits.

\color{black}
\begin{tw}\label{closedormonotone}
    Let $\seq{A, p} \in \Ult^*_\fin$ be a partial automorphism.
    For every orbit $p^\Z(a)$ of the point sort we can extend $\seq{A, p}$ to $\seq{B, q} \in \Ult^*_\fin$ so that the extended orbit $q^\Z(a)$ is closed or monotone.
\end{tw}
\begin{proof}
    Let $\seq{X, f} \geq \seq{A, p}$ be any extension to a total dc-automorphism.
    If $f^\Z(a)$ is closed or monotone, we are done as we can take a suitable finite restriction $\seq{B, q} \geq \seq{A, p}$.
    Otherwise, $f^\Z(a)$ is horizontal.
    Let $a_n = f^n(a)$ for $n \in \Z$, let $(n_i)_{i \in I}$ be the constructing sequence (finite or infinite) for $f^\Z(a)$, and let $u_i = d(a_0, a_{n_i})$ for every $i \in I$ (so $\set{u_i: i \in I}$ is the set of all positive distances used in the orbit $f^\Z(a)$).

    Suppose there is an interval $[M, N] \subs \Z$ such that $f^\Z(a) \cap A \subs \set{a_M, \ldots, a_N}$ and $d(a_M, a_N) \leq d(a_M, b), d(a_N, b)$ (and so $d(a_M, b) = d(a_N, b)$) for every $b \in B \setminus \set{a_M, a_N}$ where $B = A \cup \set{a_M, \ldots, a_{N - 1}}$.
    Then we can put $D_B = D_A \cup d_X\im{B \times B}$ and $\seq{B, q_0} = \seq{X, f}\restr{\seq{B, D_B}}$.
    We have $D_f(d(a_{N - 1}, b)) = d(a_N, f(b)) = d(a_M, f(b))$ for every $b \in \dom(q_0) \setminus \set{a_{N - 1}}$.
    For $b = a_{N - 1}$ we have $D_f(d(a_{N - 1}, b)) = 0 = d(a_M, a_M)$.
    It follows that $q = q_0 \cup \set{a_{N - 1} \mapsto a_M}$ with $D_q = D_{q_0}$ is a partial automorphism of $B$ with a closed orbit $q^\Z(a)$.

    It remains to show that we can find a suitable interval $[M, N]$.
    Note that by the ultrametric triangle inequality, for every $b \in A \setminus f^\Z(a)$ there is at most one point $a_{m_b}$ such that $d(a_{m_b}, b) < u_i$ for every $i \in I$, and let us put $m_b = 0$ if there is no such close point.
    Let $[M_0, N_0]$ be such that $f^\Z(a) \cap A \subs \set{a_M, \ldots, a_N}$ and $\set{m_b: b \in A \setminus f^\Z(a)} \subs (M_0, N_0)$.
    
    If the constructing sequence $(n_i)_{i \in I}$ is finite, we let $i = \max(I)$ and we take any interval $[M, N] \supseteq [M_0, N_0]$ such that $n_i$ divides $N - M$.
    Then $d(a_M, a_N)$ is the smallest distance $u_i$, and as $m_b \neq M, N$ for any $b \in A \setminus f^\Z(a)$, we have $d(a_M, b), d(a_N, b) \geq u_i$ and we are done.

    If the constructing sequence $(n_i)$ is infinite, we need to be more careful.
    We take $i$ such that $K := n_i - (N_0 - M_0) \geq \card{A \setminus f^\Z(a)}$, so that $[M(k), N(k)] := [M_0 - k, M_0 - k + n_i] \supseteq [M_0, N_0]$ for every $0 \leq k \leq K$.
    In other words, we consider $K + 1 > \card{A \setminus f^\Z(a)}$ possibilities for the interval $[M, N]$ of length $n_i$, all of which lie in a common interval $[M_1, N_1]$ of length $n_{i + 1} \geq 2n_i$.
    Again by the ultrametric triangle inequality, for every $b \in A \setminus f^\Z(a)$ there is at most one $n_b \in [M_1, N_1)$ such that $d(a_{n_b}, b) < u_i$, and hence there is $0 \leq k \leq K$ such that $\set{M(k), N(k)}$ misses all the values $n_b$.
    Therefore, for $[M, N] = [M(k), N(k)]$ we have $d(a_M, b), d(a_N, b) \geq u_i = d(a_M, a_N)$ for every $b \in A \setminus f^\Z(a)$.
    Since also for every $M < n < N$ we have $d(a_M, a_n), d(a_N, a_n) > u_i$, we are done.
\end{proof}

\color{black}



\section{Amalgamation bases and determined partial automorphisms}

In this section, we present a general strategy (Strategy \ref{strategy}), which applies to nice categories of structures (Definition \ref{nice}), for showing that the class of finite partial automorphisms has the CAP.

A structure $A$ is \emph{locally finite} if for every finite subset $F \subseteq A$ the substructure of $A$ generated by $F$ is finite as well.

\begin{df}\label{nice}
    In order to collect the relevant assumptions, we say that a category of structures $\C$ is \emph{nice} if it is hereditary, $\sigma$-complete, consists of locally finite structures, and $\C_\fin$ is Fraïssé.
\end{df}

\begin{df}
  Let $\C$ be a category of structures. 
    We say that $\seq{A, p} \in \C^*_\fin$ has a \emph{unique totalization} (or is \emph{determined}) if there is an extension $\seq{A, p} \to \seq{\hat{A}, \hat{p}}$ (in $\C^*$) to a total automorphism such that every extension $\seq{A, p} \to \seq{X, f}$ to a total automorphism (in $\C^*$) admits a unique factorization $\seq{\hat{A}, \hat{p}} \to \seq{X, f}$.
    In other words, the orbits of $\seq{A, p}$ can be completed to full orbits in a unique way (up to an isomorphism).
    The extension $\seq{\hat{A}, \hat{p}}\in \C^*$ we call the \emph{unique totalization} (of $\seq{A, p}$).
\end{df}

The concept of a determined partial automorphism appears already in Kuske--Truss \cite{KuTr}, where they show the existence of a generic automorphism of the universal partial order. The idea of extending a partial automorphism to a determined one was further used by Kaplan--Rzepecki--Siniora \cite{KaRzSi}, who show that there is a generic automorphism of the universal meet-tree. They point out that amalgamation bases are determined (\cite[Proposition 4.2]{KaRzSi}) and have an analog of our Lemma \ref{thm:pointwise_determined}  (\cite[Lemma 4.11]{KaRzSi}). In their specific situation amalgamation bases correspond to determined partial automorphisms (\cite[Corollary 4.19]{KaRzSi}).
Moreover, in order to show the CAP for finite meet-trees with partial automorphisms they show amalgamation of infinite meet-trees with total automorphisms (\cite[Lemma 4.16]{KaRzSi}). That corresponds to our Corollary \ref{thm:dc_auto_amalgamation}.
We only found out about the ideas used in their article after completing our proofs. We further develop those ideas into a general strategy usable for a nice category of structures (Strategy \ref{strategy}).

\begin{observation} \label{obs:totalization_exists}
Note that if $\C$ is nice then every $\seq{A, p} \in \C_\fin^*$ admits an extension to a total automorphism $f$ of a structure $X \in \sigma\C$, namely we let $X$ be the Fraïssé limit and we apply the homogeneity.
\end{observation}

\begin{lm} \label{thm:amalgamation_base_determined}
    Let $\C$ be a nice category of structures.
    If $\seq{A, p} \in \C^*_\fin$ is an amalgamation base, then it is determined.
\end{lm}
\begin{proof}
    Let $\seq{A, p} \leq \seq{X, f}$ be any extension to a total automorphism and let $\set{\seq{a^i_0, \ldots, a^i_{n_i}}\colon i \in I}$ be an enumeration of the orbits (of all of the sorts) of $\seq{A, p}$.
    We take $\seq{\hat{A}, \hat{p}}$ to be the substructure of $\seq{X, f}$ generated by the set $\set{f^k(a^i_0): k \in \Z, i \in I}$ of elements of the full orbits.
    Note that $\seq{\hat{A}, \hat{p}}$ is a total automorphism as $\hat{A}$ is $f$-invariant.
    Since $\C$ is hereditary, $\seq{\hat{A}, \hat{p}} \in \C^*$.

    We show that for every other total extension $e\maps \seq{A, p} \to \seq{Y, g}$ there is a unique $\C^*$-map $\seq{\hat{A}, \hat{p}} \to \seq{Y, g}$ extending $e$.
    For every $n \in \N$ let $A_n \leq \hat{A}$ be generated by $\set{f^k(a^i_0): k \in [-n, n], i \in I}$ and let $B_n \leq Y$ be generated by $\set{g^k(e(a^i_0)): k \in [-n, n], i \in I}$.
    Moreover, we put $p_n = \hat{p}\restr{A_n}$ and $q_n = g\restr{B_n}$ and we let let $n_0 \in \N$ be such that $A \leq A_{n_0}$ and $e\im{A} \leq B_{n_0}$.
    Since $\C$ is hereditary and locally finite, $A_n, B_n \in \C_\fin$.
    Since $\seq{A, p}$ is an amalgamation base, the extensions $\seq{A, p} \leq \seq{A_n, p_n}$ and $e_n\maps \seq{A, p} \to \seq{B_n, q_n}$ can be amalgamated in $\C_\fin^*$, and necessarily $f^k(a^i_0)$ and $g^k(e(a^i_0))$ for $k \in [-n, n]$ and $i \in I$ need to be identified in the amalgamation. Therefore we have a unique isomorphism $\seq{A_n, p_n} \to \seq{B_n, q_n}$ and a unique embedding $\seq{A_n, p_n} \to \seq{Y, g}$ extending~$e$.
    Together, we have a unique embedding $\seq{\hat{A}, \hat{p}} \to \seq{Y, g}$ extending $e$.
\end{proof}

\begin{tw} \label{thm:amalgamation_from_totalization}
    Let $\C$ be a nice category of structures such that the full subcategory of $\C^*$ consisting of total automorphisms has AP.
    Then $\C^*_\fin$ has CAP if and only if determined partial automorphisms are cofinal in $\C^*_\fin$.
    In that case, $\seq{A, p} \in \C^*_\fin$ is an amalgamation base if and only if it is determined.
\end{tw}
\begin{proof}
    One implication follows from Lemma~\ref{thm:amalgamation_base_determined}.
    For the other implication suppose that determined partial automorphisms are cofinal.
    Let $\seq{A, p} \in \C^*_\fin$ be determined. We show that $\seq{A, p}$ is an amalgamation base. 
    Let $f_i\maps \seq{A, p} \to \seq{B_i, q_i}$ for $i = 1, 2$ be an extension in $\C^*_\fin$.
    Since determined partial automorphisms are cofinal, we can without loss of generality assume that $\seq{B_i, q_i}$ are determined.
    We will construct maps as in Figure~\ref{fig:amalgamation_from_totalization}.
    Let $e_A\maps \seq{A, p} \to \seq{\hat{A}, \hat{p}}$ and $e_{B_i}\maps \seq{B_i, q_i} \to \seq{\hat{B_i}, \hat{q_i}}$ be the unique totalizations.
    Let $\hat{f_i}\maps \seq{\hat{A}, \hat{p}} \to \seq{\hat{B_i}, \hat{q_i}}$ be the unique $\C^*$-maps such that $\hat{f_i} \cmp e_A = e_{B_i} \cmp f_i$.
    Let $\seq{Z, h} \in \C^*$ be a total automorphism realizing an amalgamation of the maps $\hat{f_i}$, i.e. there are maps $g_i\maps \seq{\hat{B_i}, \hat{q_i}} \to \seq{Z, h}$ such that $g_1 \cmp \hat{f_1} = g_2 \cmp \hat{f_2}$.
    Then we have $g_1 \cmp e_{B_1} \cmp f_1 = g_2 \cmp e_{B_2} \cmp f_2$, which is an amalgamation in $\C^*$.
    To obtain an amalgamation in $\C_\fin^*$ we just restrict to the substructure of $Z$ generated by $g_1\im{e_{B_1}\im{B_1}} \cup g_2\im{e_{B_2}\im{B_2}}$, which is in $\C_\fin$ as $\C$ is hereditary and locally finite.
\end{proof}

\begin{figure}[!ht]
    \centering
\begin{tikzpicture}[
    x = {(10em, 0em)},
    y = {(0em, 4em)},
    label/.style = {
        edge label=#1,
        font=\footnotesize,
    }
]
    \node (A) at (0, 0) {$\seq{A, p}$};
    \node (B1) at (-1, 1) {$\seq{B_1, q_1}$};
    \node (B2) at (+1, 1) {$\seq{B_2, q_2}$};
    
    \node (Ah) at (0, 1) {$\seq{\hat{A}, \hat{p}}$};
    \node (B1h) at (-1, 2) {$\seq{\hat{B_1}, \hat{q_1}}$};
    \node (B2h) at (+1, 2) {$\seq{\hat{B_2}, \hat{q_2}}$};

    \node (Z) at (0, 3) {$\seq{Z, h}$};
    \node (C) at (0, 2) {$\seq{Z, h}\restr{\seq{B_1 \cup B_2}}$}; 

    \graph{
        (A) ->[label=$f_1$] (B1),
        (A) ->[label=$f_2$, swap] (B2),
        (A) ->[label=$e_A$] (Ah),
        (Ah) ->[label=$\hat{f}_1$, pos=0.2, inner sep=2pt] (B1h),
        (Ah) ->[label=$\hat{f}_2$, pos=0.2, inner sep=2pt, swap] (B2h),
        (B1) ->[label=$e_{B_1}$] (B1h),
        (B2) ->[label=$e_{B_2}$, swap] (B2h),
        (B1h) ->[label=$g_1$] (Z),
        (B2h) ->[label=$g_2$, swap] (Z),
        (C) ->[dashed] (Z),
        {(B1), (B2)} ->[dashed] (C),
    };
\end{tikzpicture}
    \caption{Amalgamation using unique totalizations.}
    \label{fig:amalgamation_from_totalization}
\end{figure}
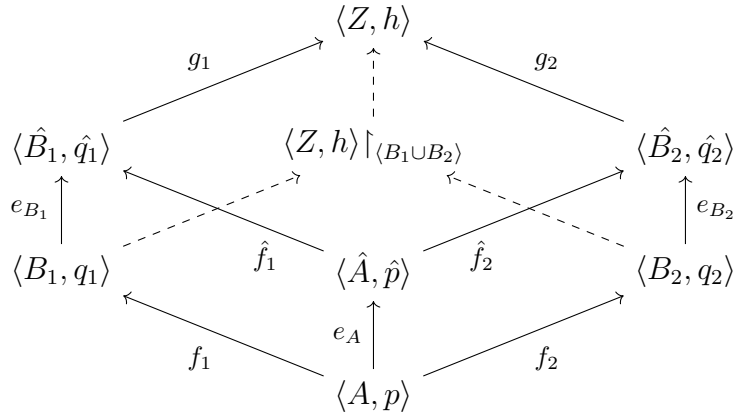

In order to detect determined partial automorphisms and to show that they are cofinal, it is useful to work with one-point orbit extensions.

\begin{df}
Let $\C$ be a category of structures.
A \emph{forward one-point orbit extension} of $\seq{A, p} \in \C^*$ at $x \in A \setminus \dom(p)$ is an extension $\seq{B, q} \geq \seq{A, p}$ in $\C^*$ such that $x \in \dom(q)$, $B$ is generated by $A \cup \set{q(x)}$, $\dom(q)$ is generated by $\dom(p) \cup \set{x}$, $\rng(q)$ is generated by $\rng(p) \cup \set{q(x)}$.
Note that in the case of relational structures, we just extend the domain and the range by one point.
Also note that we may deal with many-sorted structures, so every point has the associated sort.

We say that two forward one-point orbit extensions $\seq{B_1, q_1}, \seq{B_2, q_2} \geq \seq{A, p}$ at $x$ are {\it isomorphic}, if $\id{A} \cup \{q_1(x) \mapsto q_2(x)\}$ extends to a (necessarily unique) isomorphism $\seq{B_1, q_1} \to \seq{B_2, q_2}$.
When we say that one-point orbit extensions are unique, we mean unique up to isomorphism.

Analogously we define \emph{backward one-point orbit extensions} at $x \in A \setminus \rng(p)$ and their isomorphisms.
In particular $\seq{A, p} \leq \seq{B, q}$ is a backward one-point orbit extension at $x$ if and only if $\seq{A, p^{-1}} \leq \seq{B, q^{-1}}$ is a forward one-point orbit extension at $x$.
\end{df}

\begin{observation} \label{thm:determined_to_pointwise}
    Let $\C$ be a nice category of structures.
    If $\seq{A, p} \in \C^*_\fin$ is determined, then it has unique one-point orbit extensions at all points, and these extensions are themselves determined.
\end{observation}
\begin{proof}
    Let $\seq{B_1, q_1}$ and $\seq{B_2, q_2}$ be (without loss of generality forward) one-point orbit extensions of $\seq{A, p}$ at $x$.
    For $i = 1, 2$, let $\seq{X_i, f_i} \in \C^*$  be an extension of $\seq{B_i, q_i}$ to a total automorphism such that $X_i$ is generated by the orbits $f_i^\Z(b)$ for $b \in B_i$.
    We first take any extension to a total automorphism, which exists since $\C_\fin$ is Fraïssé, and then we restrict to the substructure generated by the orbits of the points in $B_i$.
    The restriction is in $\C^*$ since $\C$ is nice.
    It follows that both $\seq{X_1, f_1}$ and $\seq{X_2, f_2}$ are isomorphic to the unique totalization of $\seq{A, p}$, so there is a unique isomorphism $f\maps \seq{X_1, f_1} \to \seq{X_2, f_2}$ extending $\id{A}$.
    Hence, $f(q_1(x)) = q_2(x)$ and the restriction of $f$ gives the isomorphism of the one-point extensions $\seq{B_1, q_1}$ and $\seq{B_2, q_2}$.
    Moreover, $\seq{X_i, f_i}$ is the unique totalization of $\seq{B_i, q_i}$ for $i = 1, 2$.
\end{proof}

The following lemma and proposition can be viewed as a converse to the observation above.

\begin{lm} \label{thm:pointwise_determined}
    Let $\C$ be a nice category of structures and let $(\seq{A_n, p_n})_n$ be a countable increasing chain (finite or infinite) of one-point orbit extensions in $\C^*_\fin$, i.e. $\seq{A_{n+1}, p_{n+1}}$ is a forward or backward one-point orbit extension of $\seq{A_n, p_n}$ at $x_n \in A_n$, and suppose that $\seq{\hat{A}, \hat{p}} = \bigcup_n \seq{A_n, p_n}$ is a total automorphism.
    Then $\seq{\hat{A}, \hat{p}}$ is the unique totalization of $\seq{A_0, p_0}$ if and only if every $\seq{A_{n + 1}, p_{n + 1}}$ is the unique forward/backward one-point orbit extension of $\seq{A_n, p_n}$ at $x_n$.
\end{lm}
\begin{proof}
    The ``only if'' part follows from Observation~\ref{thm:determined_to_pointwise}.
    For the ``if'' part let $e_0\maps \seq{A_0, p_0} \to \seq{X, f}$ be any $\C^*$-map to a total automorphism.
    We want to show that there is a unique extension $\hat{e}\maps \seq{\hat{A}, \hat{p}} \to \seq{X, f}$.
    We proceed by induction.
    Suppose that we already have an extension $e_n\maps \seq{A_n, p_n} \to \seq{X, f}$ of $e_0$.
    We extend it to $e_{n + 1}\maps \seq{A_{n + 1}, p_{n + 1}} \to \seq{X, f}$.
    Let $\seq{B, q} \leq \seq{X, f}$ be the forward/backward one-point orbit extension of the copy of $\seq{A_n, p_n}$ at $e_n(x_n)$ coming from the restriction of $\seq{X, f}$.
    Then, $e_n\maps \seq{A_n, p_n} \to \seq{B, q}$ is a forward/backward one-point orbit extension at $x_n$, which is by uniqueness isomorphic to $\seq{A_{n+1}, p_{n+1}}$, so we let $e_{n + 1}\maps \seq{A_{n + 1}, p_{n + 1}} \to \seq{B, q} \leq \seq{X, f}$ be the isomorphism.
    Finally, we put $\hat{e} = \bigcup_n e_n$.
    As any $\hat{e}$ needs to satisfy $\hat{e}(p_{n + 1}(x_n)) = f(\hat{e}(x_n))$ for every $n$, $\hat{e}$ is uniquely determined on $A_{n + 1}$, given it is uniquely determined on $A_n$, and therefore we get uniqueness by induction as well.
\end{proof}

\color{black}

\begin{prop}\label{prop: suffimpldeterm}
    Let $\C$ be a nice category of structures.
    Suppose that $\Su$ is a full subcategory of $\C^*_\fin$ with the following properties.
     \begin{enumerate}
        \item Partial automorphisms from $\Su$ have a unique one-point orbit extension at every point.
        \item One-point orbit extensions of partial automorphisms in $\Su$ are themselves in $\Su$.
    \end{enumerate}
    Then any $\seq{A, p}\in\Su$ is determined. 
    Consequently, if we additionally have that
    \begin{enumerate}[resume]
        \item $\Su$ is cofinal in $\C^*_\fin$,
    \end{enumerate}
    then determined partial automorphisms are cofinal in $\C^*_\fin$.
\end{prop}

\begin{proof}
    Let $\seq{A, p} \in \Su$.
    Let $\seq{X, f} \geq \seq{A, p}$ be any total automorphism in $\C^*$, which exists by Observation~\ref{obs:totalization_exists}, and let $\seq{\hat{A}, \hat{p}} \leq \seq{X, f}$ be generated by the finite union of orbits of $f$ intersecting $A$.
    We show that $\seq{\hat{A}, \hat{p}}$ is the unique totalization of $\seq{A, p}$.

    There is a countable enumeration $(a_n)$ of $\hat{A}$ (finite or infinite) such that $A = \set{a_n: n \leq n_0}$ for some $n_0$ and such that for every $n \geq n_0$ we have $a_{n + 1} = \hat{p}^{\eps_n}(x_n)$ for some $\eps_n \in \set{-1, 1}$ and $x_n \in A_n$ where $A_n \leq \hat{A}$ is the substructure generated by $\set{a_i: i \leq n}$.
    If we put $p_n = \hat{p}\restr{A_n}$, we have that for every $n$, either $\seq{A_{n + 1}, p_{n + 1} }\geq \seq{A_n, p_n}$ is a forward/backward one-point orbit extension at $x_n$, or $\seq{A_{n + 1}, p_{n + 1}} = \seq{A_n, p_n}$.
    By passing to a subsequence if needed, we may assume that we always have one-point extensions.
    
    By Condition~(2), $\seq{A_n, p_n} \in \Su$ for every $n \geq n_0$, and so by Condition~(1), $\seq{A_{n + 1}, p_{n + 1}}$ is the unique forward/backward one-point orbit extension at $x_n$.
    Hence, by Lemma~\ref{thm:pointwise_determined}, $\seq{\hat{A}, \hat{p}}$ is the unique totalization of $\seq{A, p}$.
    If we additionally have Condition~(3), then clearly determined partial automorphisms are cofinal in $\C^*_\fin$.
\end{proof}

\begin{strategy} \label{strategy}
Now we are ready to describe our strategy for proving CAP and simultaneously characterizing amalgamation bases in $\C^*_\fin$ for a nice category of structures~$\C$.
\begin{enumerate}
    \item We show that the full subcategory of $\C^*$ consisting of total automorphisms has AP, so that we can apply Theorem~\ref{thm:amalgamation_from_totalization}.
    \item We pick a full subcategory $\Su \subseteq \C^*_\fin$ of so-called sufficient partial automorphisms satisfying the assumptions of Proposition~\ref{prop: suffimpldeterm}.
    \item Then by Proposition~\ref{prop: suffimpldeterm} and Theorem~\ref{thm:amalgamation_from_totalization} it follows that $\C^*_\fin$ has CAP, the amalgamation bases are exactly the determined partial automorphisms, and $\Su$ forms a cofinal family of those.
    \item If we also show that every amalgamation base / determined partial automorphism is necessarily in $\Su$, then $\Su$ forms exactly the class of amalgamation bases / determined partial automorphisms.
\end{enumerate}
\end{strategy}

The following notion often gives a necessary condition for being an amalgamation base, and in our application it even turns out to be sufficient.

\begin{df}
For a category of structures $\C$ and a partial automorphism $\seq{A, p} \in \C^*$ we say that its orbit $p^\Z(a)$ is \emph{stable} if it does not merge with another orbit of $p$ and does not become closed in any extension:
for every $\seq{B, q} \geq \seq{A, p}$ in $\C^*$ we have $q^\Z(a) \cap A = p^\Z(a)$, and if $q^\Z(a)$ is closed, then already $p^\Z(a)$ is closed.
\end{df}

\begin{prop} \label{thm:stable_necessary}
    Let $\C$ be a nice category of structures and suppose that $\C_\fin$ has the strong amalgamation property.
    If $\seq{A, p} \in \C^*_\fin$ is an amalgamation base, then every orbit of $\seq{A, p}$ is stable.
\end{prop}
\begin{proof}
    Suppose that some extension $\seq{B, q} \geq \seq{A, p}$ in $\C^*_\fin$ merges or closes some orbits of $\seq{A, p}$, i.e. there is a sequence $a \mapsto t_1 \mapsto \cdots \mapsto t_n \mapsto a'$ in $q$ such that $a \in A \setminus \dom(p)$, $a' \in A \setminus \rng(p)$, and $t_i \in B \setminus A$.
    By passing to a minimal such subsequence we can assume that $t_i \notin A$ for every $i = 1, \ldots, n$.
    (The case $n = 0$ also captures the situation of $a = a'$ becoming a fixed point.)
    Since $\C$ is nice, we can pass to substructures, so without loss of generality, $\seq{B, q}$ is generated by $\seq{A \cup \{t_1, \ldots, t_n\}, p \cup \{a \mapsto t_1 \mapsto \cdots \mapsto t_n \mapsto a'\}}$, i.e. $B$ is generated by $A \cup \{t_1, \ldots, t_n\}$, and $\dom(q)$ is generated by $\dom(p) \cup \{a, t_1, \ldots, t_n\}$ (and so  $\rng(q)$ is generated by $\rng(p) \cup \{t_1, \ldots, t_n, a'\}$).

    We will construct another extension $\seq{B', q'} \geq \seq{A, p}$ not amalgamable with $\seq{B, q}$.
    Let $R \leq B$ be the substructure generated by $\rng(p) \cup \{t_1, \ldots, t_n\}$, so $\rng(q) \geq R$ is generated by $R \cup \{a'\}$.
    Without loss of generality $B\leq U$ where $U$ is the Fraïssé limit of $\C_\fin$.
    We want to find a point $x \in U \setminus B$ such that $\tp(x/R) = \tp(a'/R)$, by which we mean such that there is an isomorphism $f\maps\rng(q) \to Q$ extending $\id{R} \cup \{a' \mapsto x\}$, where $Q \leq U$ is the substructure generated by $R \cup \{x\}$.
    We take a strong amalgam $C \in \C_\fin$ of $B$ and $\rng(q)$ over $R$, view it as an extension of $B$, which we realize in $U$.
    Hence we obtain an embedding $e\maps\rng(q) \to U$ extending $\id{R}$ such that $x := e(a') \notin B$.
    
    Let $f\maps \rng(q) \to Q$ be the isomorphism witnessing $\tp(x/R) = \tp(a'/R)$, let $q'\maps \dom(q) \to Q$ be the composition $f \cmp q$, and let $B' \leq U$ be the substructure generated by $A \cup \dom(q') \cup \rng(q')$.
    Then $\seq{B', q'}$ is an extension of $\seq{A, p}$ in $\C^*_\fin$ that cannot be amalgamated with $\seq{B, q}$ over $\seq{A, p}$.
    This is because $s^{n + 1}(a) =a' \in A$ for every extension $s \supseteq q$, while $(s')^{n+1}(a)=x \not \in A$ for every extension $s' \supseteq q'$.
\end{proof}

We shall illustrate Strategy~\ref{strategy} on the example of linear orders.
Then we proceed to our main application – ultrametric spaces with partial dc-automorphisms.

\begin{prop}\label{lo_totaliz}
    $(\Lin^0)^*_\fin$  has CAP and for $\seq{A,p}\in (\Lin^0)^*_\fin$ the following conditions are equivalent:
    \begin{enumerate}[label={\rm(\alph*)}]
        \item $\seq{A,p}$ is an amalgamation base;
        \item $\seq{A,p}$ is determined;  
        \item every orbit of $\seq{A,p}$ is stable.
    \end{enumerate}
\end{prop}

\begin{proof}
We follow the strategy above.
First, by Lemma~\ref{thm:order_auto_amalagamation}, which we will prove later, total automorphisms have AP.
Second, we let class $\Su \subseteq (\Lin^0)^*_\fin$ of sufficient partial automorphisms consist exactly of those with all orbits stable.
We need to check the assumptions (1), (2), (3) of Proposition~\ref{prop: suffimpldeterm}.

Clearly, (2) is satisfied since extending a stable orbit cannot make it unstable.
Also we have (3) since if some orbit(s) of a given partial automorphism $\seq{A, p}$ can be closed or merged in an extension $\seq{B, q} \geq \seq{A, p}$, we can pass to that extension, which without loss of generality does not add any new orbits (we can always remove new orbits).
After finitely many steps all orbits become stable.

Now we show (1), the uniqueness of one-point extensions, i.e. for every $\seq{A, p} \in \Su$ and $x \in A \setminus \dom(p)$ there is up to isomorphism a unique forward one-point orbit extension of $\seq{A, p}$ at $x$ in $(\Lin^0)^*_\fin$.
(It is enough to consider only forward one-point orbit extensions as $\seq{A, p} \in \Su$ if and only if $\seq{A, p^{-1}} \in \Su$.)
So let $x \in A \setminus \dom(p)$.
Forward one-point orbit extensions are exactly of the form $\seq{A \cup \{y\}, p \cup \{x \mapsto y\}}$ for $y \notin \rng(p)$ and $p[A_0] < y< p[A_1]$, where $A_0 = \{a \in \dom(p): a < x\}$ and $A_1 = \{a \in \dom(p): x < a\}$.
Note that the sets $A_i$ may be empty and the point $y$ may or may not be in $A$.
However, if there is any $y \in A$ such that $p[A_0] < y < p[A_1]$, then the corresponding extension would close or merge orbits, which would contradict the assumption.
Hence, the interval between $p[A_0]$ and $p[A_1]$ is empty in $A$ and can be filled with a new point in a unique way up to isomorphism.

Since the assumptions (1), (2), (3) are satisfied, by Proposition~\ref{prop: suffimpldeterm} and Theorem~\ref{thm:amalgamation_from_totalization} we have that $(\Lin^0)^*_\fin$ has CAP and we have (a)~$\Leftrightarrow$~(b)~$\Leftarrow$~(c) in the present proposition.
The implication (a)~$\Rightarrow$~(c) follows from Proposition~\ref{thm:stable_necessary}.
\end{proof}

\section{A generic dc-automorphism}

We apply Strategy~\ref{strategy} to characterize amalgamation bases and show CAP of $\Ult^*_\fin$ (Theorem~\ref{spaces_totaliz}), and consequently to prove the existence of a generic dc-automorphism of $\U$ (Corollary~\ref{cor:generic_auto}).

\subsection{Total dc-automorphisms have amalgamation}

We shall prove that the full subcategory of $\Ult^*$ consisting of total automorphisms has AP.
This is the main assumption of Theorem~\ref{thm:amalgamation_from_totalization}.

The first step is to give a certain ``free-like'' amalgamation of (even infinite) ultrametric spaces – a strong amalgamation such that distances between points from the extensions are as large as possible.
This is essentially an instance of so-called \emph{independent amalgamation} (see \cite[Section~3]{Solecki05}, \cite[Section~2]{DouchaMalicki19}, \cite[Definition~2.3]{Sabok19}, and Remark~\ref{rmk:independent_amalgamation}), which is well-suited to independently amalgamate new orbits added by the respective extensions of automorphisms.

\medskip

Recall that $\UltIso^E \subs \Ult$ for $E \in \Lin^0$ denotes the subcategory of all ultrametric spaces $X$ with $D_X \leq E$ and of all isometric embeddings.

\begin{lm} \label{thm:complete_isometric_amalgamation}
    Let $E$ be a complete linear order in $\Lin^0$ (i.e. every subset has the infimum and the supremum) where $0$ is isolated, i.e. $\min(E \setminus \set{0})$ exists.
    Then $\UltIso^E$ has AP.
    More precisely, for all spaces $X \leq Y_i$ in $\UltIso^E$, $i = 1, 2$, with $Y_1 \cap Y_2 = X$, the following ultrametric on $Z = Y_1 \cup Y_2$ with $D_Z = E$ extends $d_{Y_1} \cup d_{Y_1}$:
    \[
        d_Z(x, y) = \inf_{a \in Y_i \cap Y_j} \max\set{d_{Y_i}(x, a), d_{Y_j}(a, y)}
    \]
    for every $x \in Y_i$ and $y \in Y_j$.
\end{lm}
\begin{proof}
    As $d_{Y_i}(x, y) = \inf_{a \in A} \max\set{d_{Y_i}(x, a), d_{Y_i}(a, y)}$ for every $x, y \in Y_i$ and $A \subseteq Y_i$ containing $x$ or $y$, the map $d_Z$ extends $d_{Y_i}$ and is well-defined (if $x \in X$, it does not matter whether we view it as a member of $Y_1$ or $Y_2$).
    Clearly, $d_Z$ is symmetric and $d_Z(x, x) = 0$.
    Since $0 \in E$ is isolated, if $d_Z(x, y) = 0$, then $\max\set{d_{Y_i}(x, a), d_{Y_j}(a, y)} = 0$ for some $a$, and so $x = a = y$.
    
    It remains to prove the ultrametric triangle inequality, i.e. for every $x \in Y_i$, $y \in Y_j$, and $z \in Y_k$ we want
 
    \begin{align*}
        \inf_{a \in Y_i \cap Y_k} \max\set{d_{Y_i}(x, a), d_{Y_k}(a, z)}
        \leq
        \max\bigl\{&\inf_{b \in Y_i \cap Y_j} \max\set{d_{Y_i}(x, b), d_{Y_j}(b, y)} 
        ,\\
        &\inf_{c \in Y_j \cap Y_k} \max\set{d_{Y_j}(y, c), d_{Y_k}(c, z)}\bigr\}.
    \end{align*}
    
    Since every complete linear order is completely distributive, we may rewrite the right-hand side as 
      \[
        \inf_{b \in Y_i \cap Y_j} \inf_{c \in Y_j \cap Y_k} \max\{d_{Y_i}(x, b) , d_{Y_j}(b, y) , d_{Y_j}(y, c) , d_{Y_k}(c, z)\}.
    \]
    
    Since $\max\{d_{Y_j}(b, y) , d_{Y_j}(y, c)\} \geq d_{Y_j}(b, c)$, it is enough to show
    
     \[
        \inf_{a \in Y_i \cap Y_k} \max\{d_{Y_i}(x, a) , d_{Y_k}(a, z)\}
        \leq
        \inf_{b \in Y_i \cap Y_j} \inf_{c \in Y_j \cap Y_k} \max\{d_{Y_i}(x, b) , d_{Y_j}(b, c) , d_{Y_k}(c, z)\}.
    \]
    
    If $i = j$, the right-hand side is greater than $\inf_{b \in Y_i,\, c \in Y_i \cap Y_k} \max\{d_{Y_i}(x, c) , d_{Y_k}(c, z)\}$, and we are done (if $Y_i = \emptyset$, the desired inequality is trivially true).
    If $j = k$, analogous argument works.
    Otherwise, we have $i \neq j \neq k = i$, so $Y_i \cap Y_j = Y_j \cap Y_k = X$, and the inequality simplifies to
     \[
        \inf_{a \in Y_i} \max\{d_{Y_i}(x, a) , d_{Y_i}(a, z)\}
        = d_{Y_i}(x, z) \leq
        \inf_{b, c \in X} \max\{d_{Y_i}(x, b) , d_{X}(b, c) , d_{Y_i}(c, z)\}.
        \qedhere
    \]
\end{proof}

Let $\Ult^{*\seq{E, \phi}} \subseteq \Ult^*$ for $E \in \Lin^0$ and $\phi \in \aut(E)$ denote the subcategory whose objects $\seq{X, p}$ are such such that $\seq{D_X, D_p} \leq \seq{E, \phi}$ and whose morphisms are $\Ult^*$-maps $f\maps \seq{X, p} \to \seq{Y, q}$ that are isometric embeddings, i.e. $D_f \subseteq \id{E}$.

\begin{prop} \label{thm:iso_auto_amalgamation}
    Let $E \in \Lin^0$ be a complete linear order where $0$ is isolated, and let $\phi \in \aut(E)$.
    Then the full subcategory of $\Ult^{*\seq{E, \phi}}$ consisting of total automorphisms has AP.
\end{prop}
\begin{proof}
    Let $\seq{X, f} \leq \seq{Y_i, g_i}$ for $i = 1, 2$ be ultrametric spaces with automorphisms such that $D_X \leq D_{Y_i} \leq E$ and $D_f \subseteq D_{g_i} \subseteq \phi$.
    Without loss of generality $Y_1 \cap Y_2 = X$.
    Let $Z = \seq{Y_1 \cup Y_2, d_Z, E}$ be the free-like amalgamation obtained from Lemma~\ref{thm:complete_isometric_amalgamation}.
    As $g_i \in \aut(Y_i)$ and $f = g_1 \cap g_2 \in \aut(X)$, the map $h = g_1 \cup g_2$ is a well-defined bijection on $Z$.

    It remains to show that $\seq{h, \phi} \in \aut(Z)$.
    For every $x \in Y_i$ and $y \in Y_j$ we have
    \begin{align*}
        \phi(d_Z(x, y)) &= \phi\bigl(\inf_{a \in Y_i \cap Y_j} \max\set{d_{Y_i}(x, a),\, d_{Y_j}(a, y)}\bigr) \\
        &= \inf_{a \in Y_i \cap Y_j} \max\set{\phi(d_{Y_i}(x, a)),\, \phi(d_{Y_j}(a, y))} \\
        &= \inf_{a \in Y_i \cap Y_j} \max\set{d_{Y_i}(g_i(x), g_i(a)),\, d_{Y_j}(g_j(a), g_j(y))} \\
        &= \inf_{a \in \rng(g_i \cap g_j) = Y_i \cap Y_j} \max\set{d_{Y_i}(h(x), a),\, d_{Y_j}(a, h(y))} 
        = d_Z(h(x), h(y)).
        \qedhere
    \end{align*}
\end{proof}

\begin{uwgi} \label{rmk:independent_amalgamation}
    As seen in the proof of the previous proposition, the amalgamation from Lemma~\ref{thm:complete_isometric_amalgamation} is independent in the sense that it is such an amalgamation $Y_1 \to Z \from Y_2$ of a span $Y_1 \from X \to Y_2$ in $\UltIso^E$ that every lifting $\seq{Y_1, g_1} \from \seq{X, f} \to \seq {Y_2, g_2}$ with respect to the forgetful functor $\Ult^{*\seq{E, \phi}} \to \UltIso^E$ admits an amalgamating lifting $\seq{Y_1, g_1} \to \seq{Z, h} \from \seq{Y_2, g_2}$.
    This can be compared to the notion of \emph{amalgamation property with automorphisms (APA)} coined by Siniora~\cite{Siniora} and originating from Hodges–Hodkinson–Lascar–Shelah~\cite{HHLS}.
\end{uwgi}

To elevate the previous proposition from $\Ult^{*\seq{E, \phi}}$ to $\Ult^*$, we will need the following.

\begin{observation} \label{obs:liftings}
    For every $\seq{X, D_X} \in \Ult$ and every $\Lin^0$-map $f\maps D_X \to E$ there is the canonical lifting: the dc-embedding $\bar{f}\maps \seq{X, D_X} \to \seq{X, E}$ with $D_{\bar{f}} = f$, i.e. we put $d_{\seq{X, E}}(x, y) = f(d_X(x, y))$ and we let $\bar{f}\maps X \to X$ be the identity.
    
    The canonical liftings are compositional in the sense that for every dc-embedding $e\maps X \to Y$ and $\Lin^0$-maps $f\maps D_X \to E$ and $g\maps D_Y \to E$ such that $f = g \cmp D_e$ we have $e \cmp \bar{f} = \bar{g} \cmp e$ where $e$ is also viewed as an isometric embedding $\seq{X, E} \to \seq{Y, E}$.
    
    Analogously, we have canonical liftings for partial automorphisms, i.e. for every $\seq{X, p, D_X, D_p} \in \Ult^*$ and every $(\Lin^0)^*$-map $f\maps \seq{D_X, D_p} \to \seq{E, \phi}$ we have the lifting $\seq{X, p, D_X, D_p} \to \seq{X, p, E, \phi}$.
\end{observation}

\begin{lm} \label{thm:order_auto_amalagamation}
    The full subcategory of $(\Lin^0)^*$ consisting of total automorphisms has AP.
\end{lm}
\begin{proof}
    We can equivalently work in $\Lin^*$.
    Let $f_i\maps \seq{X, p} \to \seq{Y_i, q_i}$ for $i = 1, 2$ be embeddings of automorphisms of linear orders.
    Without loss of generality, $f_1$ and $f_2$ are inclusions, and $Y_1 \cap Y_2 = X$.
    Let $Y'_i = Y_i \setminus X$ and let $q'_i=q_i\restriction Y'_i \in \aut(Y'_i)$ for $i = 1, 2$.
    We put $Z = X \cup Y_1 \cup Y_2$ and $h = p \cup q'_1 \cup q'_2$.
    Clearly $\seq{Z, h}$ is an amalgamation on the level of sets and automorphisms of sets.
    We need to extend the linear order to $Z$ in a way that $h$ is an automorphism.

    Let $y_i \in Y'_i$ for $i = 1, 2$.
    If there is $x \in X$ such that $y_1 <_{Y_1} x <_{Y_2} y_2$, we put $y_1 <_z y_2$.
    This is coherent with the definition of $h$ since then
    \[
        h(y_1) = q_1(y_1) < q_1(x) = p(x) = q_2(x) < q_2(y_2) = h(y_2).
    \]
    We proceed analogously if there is $x \in X$ such that $y_1 >_{Y_1} x >_{Y_2} y_2$.
    Finally, if $(\from, y_1)_X = (\from, y_2)_X$, we put $y_1 <_Z y_2$.
    Then $h(y_1) <_Z h(y_2)$ since $(\from, h(y_i))_X = p\im{(\from, y_i)_X}$.
\end{proof}

\begin{wn} \label{thm:dc_auto_amalgamation}
    The full subcategory of $\Ult^*$ consisting of total automorphisms has AP.
\end{wn}
\begin{proof}
    Let $f_i\maps \seq{X, p, D_X, D_p} \to \seq{Y_i, q_i, D_{Y_i}, D_{q_i}}$ for $i = 1, 2$ be dc-embeddings of total automorphisms of ultrametric spaces, as in Figure~\ref{fig:dc-automorphism_amalgamation}.
    By Lemma~\ref{thm:order_auto_amalagamation}, there is $E \in \Lin^0$, $\phi \in \aut(E)$, and $(\Lin^0)^*$-maps $g_i\maps \seq{D_{Y_i}, D_{q_i}} \to \seq{E, \phi}$ such that $g_1 \cmp D_{f_1} = g_2 \cmp D_{f_2} =: g$.
    Without loss of generality, $E$ is complete and $0 \in E$ is isolated, as otherwise we would pass to $\seq{\bar{E}, \bar{\phi}} \geq \seq{E, \phi}$ where $\bar{E}$ is obtained by adding a new element $0 < \eps < E \setminus \set{0}$ and by taking the Dedekind completion, and where $\bar{\phi} \in \aut(\bar{E})$ is the unique extension of $\phi \cup \set{\eps \mapsto \eps}$ to the Dedekind completion.
    By Observation~\ref{obs:liftings}, we can lift the amalgamation problem to $\Ult^{*\seq{E, \phi}}$ via $\bar{g_1}$, $\bar{g_2}$, and $\bar{g}$, so it is enough to amalgamate $f_i\maps \seq{X, p, E, \phi} \to \seq{Y_i, q_i, E, \phi}$ for $i = 1, 2$. We can do that by Proposition~\ref{thm:iso_auto_amalgamation}.
\end{proof}

\begin{figure}[!ht]
    \centering
    \begin{tikzpicture}[
        x = {(7em, 0em)},
        y = {(0em, 3em)},
        label/.style = {
            edge label=#1,
            font=\footnotesize,
        },
        amalgamate/.style = {
            decorate, 
            decoration={coil,aspect=0, amplitude=1.5},
        },
    ]

    \node (X) at (0, 0) {$\seq{X, p, D_X, D_p}$};
    \node (Y1) at (-1, 1) {$\seq{Y_1, q_1, D_{Y_1}, D_{q_1}}$};
    \node (Y2) at (+1, 1) {$\seq{Y_2, q_2, D_{Y_2}, D_{q_2}}$};
    \begin{scope}[shift={(0, -0.5)}]
    \node (XE) at (0, 2) {$\seq{X, p, E, \phi}$};
    \node (Y1E) at (-1, 3) {$\seq{Y_1, q_1, E, \phi}$};
    \node (Y2E) at (+1, 3) {$\seq{Y_2, q_2, E, \phi}$};
    \node (Z) at (0, 4) {$\seq{Z, h, E, \phi}$};
    \end{scope}
    
    \graph{
        (X) ->[label=$f_1$] (Y1) ->[label=$\bar{g_1}$, dashed] (Y1E) ->[amalgamate] (Z);
        (XE) ->[label=$f_1$] (Y1E),
        (X) ->[label=$f_2$, swap] (Y2) ->[label=$\bar{g_2}$, swap, dashed] (Y2E) ->[amalgamate] (Z);
        (XE) ->[label=$f_2$, swap] (Y2E),
        (X) ->[label=$\bar{g}$, swap, dashed] (XE),
    };

    \begin{scope}[shift={(2.8, 0)}, xscale=0.6]
        \node (DX) at (0, 0) {$\seq{D_X, D_p}$};
        \node (DY1) at (-1, 1) {$\seq{D_{Y_1}, D_{q_1}}$};
        \node (DY2) at (+1, 1) {$\seq{D_{Y_2}, D_{q_2}}$};
        \node (E) at (0, 2) {$\seq{E, \phi}$};
    \end{scope}

    \graph{
        (DX) ->[label=$D_{f_1}$] (DY1) ->[label=$g_1$, dashed] (E),
        (DX) ->[label=$D_{f_2}$, swap] (DY2) ->[label=$g_2$, swap, dashed] (E),
        (DX) ->[label=$g$, swap, dashed] (E),
    };
\end{tikzpicture}
    \caption{Amalgamation of total dc-automorphisms.}
    \label{fig:dc-automorphism_amalgamation}
\end{figure}
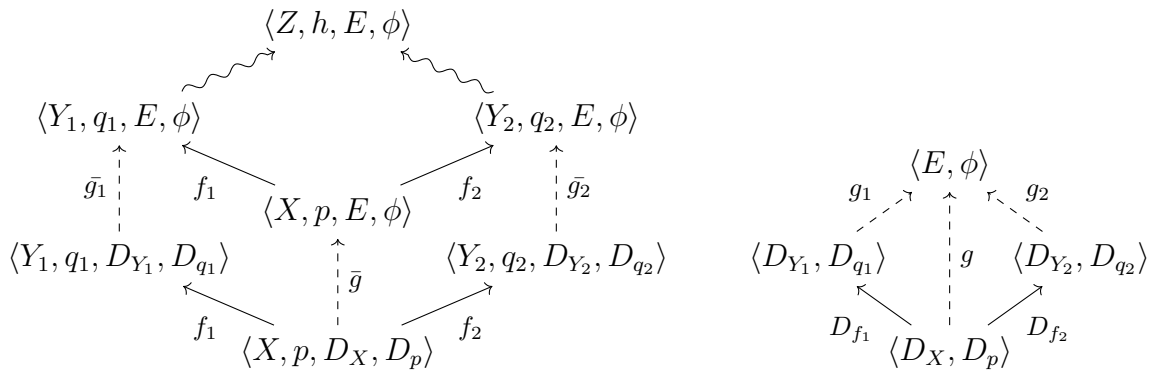

\subsection{Amalgamation bases for partial dc-automorphisms}
\label{sec: uniquetotaliz}

Let $\Su \subseteq \Ult^*_\fin$ be the class of all finite partial dc-automorphisms whose all orbits of both sorts are stable, which is a necessary condition for being an amalgamation base by Proposition~\ref{thm:stable_necessary}.
We show that it is also sufficient.

\begin{observation} \label{distance_sort_stable}
    We have that $\seq{A, p} \in \Ult^*_\fin$ is in $\Su$ if and only if every orbit of the point sort is stable and $\seq{D_A, D_p}$ is an amalgamation base in $(\Lin^0)^*_\fin$ (or equivalently is determined in $(\Lin^0)^*_\fin$).
\end{observation}
\begin{proof}
    The claim follows immediately from Proposition~\ref{lo_totaliz} once we observe that an orbit of $\seq{A, p}$ of the distance sort is stable in $\Ult^*_\fin$ if and only if it is stable in $(\Lin^0)^*_\fin$ as an orbit of $\seq{D_A, D_p}$, i.e. we can close or merge an orbit of $D_p$ in an extension $\seq{E, \phi} \geq \seq{D_A, D_p}$ if and only if we can close or merge the orbit in an extension $\seq{B, q} \geq \seq{A, p}$.
    The last equivalence is true because every such extension $\seq{B, q} \geq \seq{A, p}$ induces $\seq{E, \phi} = \seq{D_B, D_q} \geq \seq{D_A, D_p}$, and every extension $\seq{E, \phi} \geq \seq{D_A, D_p}$ induces $\seq{A, p, E, \phi} \geq \seq{A, p, D_A, D_p}$, and we let $\seq{B, q, D_B, D_q} = \seq{A, p, E, \phi}$.
\end{proof}

\begin{prop} \label{thm:unique_onepoint_extension}
    Let $\seq{A, p} \in \Su$.
    For every $x \in A \setminus \dom(p)$ there is a unique (up to isomorphism) forward one-point orbit extension $\seq{A', p'}$ at $x$.
    Moreover, every $r \in D_{A'} \setminus D_A$ is in an extension of an orbit of $D_p$.
\end{prop}

\begin{proof}
    Let $\seq{A', p'} = \seq{A \cup \set{y}, p \cup \set{x \mapsto y}}$ be any such extension.
    It is enough to show that for every $a \in A$ the distance $d_{A'}(a, y)$ is determined by the distance part $D_{p'}$ and by an old distance.
    Specifically, $d_{A'}(a, y) = \max\set{D_{p'}(r_a), s_a}$, where $r_a, s_a \in D_A$ depend only on $a$ and $x$.
    This is indeed enough since the extension $D_{p'} \geq D_p$ is unique up to isomorphism as $D_p$ is determined by Observation~\ref{distance_sort_stable}.

    Let $a \in A$.
    If $a \in \rng(p)$ and so $a = p(a')$ for some $a' \in A$, we have $d(a, y) = D_{p'}(d(a', x))$, and hence in the formula above we can  put $r_a = d(a', x)$ and $s_a = 0$.
    
    Otherwise, $a \notin \rng(p)$.
    Let $b \in \dom(p)$ be such that the distance $d(p(b), a)$ is minimal possible.
    Then we define $r_a = d(b, x)$ and $s_a = d(p(b), a)$.
    We consider the triangle $y, a, p(b)$ with distances $d(y, a)$, $s_a$, and $d(y, p(b))= D_{p'}(r_a)$.
    Clearly, if $s_a \neq D_{p'}(r_a)$, then $d(y, a) = \max\set{s_a, D_{p'}(r_a)}$ and we are done.
    Otherwise, we have $d(y, a) = s_a = D_{p'}(r_a)$ and we are also done, or $d(y, a) < s_a = D_{p'}(r_a)$.
    We claim the last case contradicts our assumptions as then $q = p \cup \set{x \mapsto a}$ induces a partial automorphism of $A$.
    For that it is enough to show that $\phi(d(x, t)) = d(a, p(t))$ for every $t \in \dom(p)$, where $\phi$ is the unique totalization of $D_{p'}$ in $(\Lin^0)^*$, as then letting $D_q$ be the restriction of $\phi$ to $d[\dom(q)^2]$ turns $q$ into a partial dc-automorphism.
    Let $t \in \dom(p)$.
    We have $d_{A'}(y, a) < s_a = d(p(b), a) \leq d(p(t), a)$.
    And therefore by considering the triangle $y, a, p(t)$ in $A'$ we have
    \[
        d(a, p(t)) = d_{A'}(y, p(t)) = D_{p'}(d(x, t)) = \phi(d(x, t)) = D_q(d(x, t)).
    \]
    The moreover part of the proposition follows from $d_{A'}(a, y) = \max\set{D_{p'}(r_a), s_a}$.
\end{proof}

\begin{tw}\label{spaces_totaliz}
    $\Ult^*_\fin$ has CAP and for $\seq{A,p}\in \Ult^*_\fin$ the following conditions are equivalent:
    \begin{enumerate}[label={\rm(\alph*)}]
        \item $\seq{A,p}$ is an amalgamation base;
        \item $\seq{A,p}$ is determined;
        \item every orbit of $\seq{A,p}$ is stable;
        \item every orbit of the point sort of $\seq{A, p}$ is stable and $\seq{D_A, D_p}$ is an amalgamation base in $(\Lin^0)^*_\fin$.
    \end{enumerate}
\end{tw}

\begin{proof}
We follow Strategy~\ref{strategy}.
First, we showed in Corollary~\ref{thm:dc_auto_amalgamation} that total automorphisms have AP.
Our class of sufficient partial automorphisms $\Su$ is defined by (c), or equivalently (d), by Observation~\ref{distance_sort_stable}.
We need to check the assumptions (1), (2), (3) of Proposition~\ref{prop: suffimpldeterm}.

Assumption (1) for the point sort follows from Proposition~\ref{thm:unique_onepoint_extension} (as in the proof of Proposition~\ref{lo_totaliz}, it is enough to consider the forward case).
For the distance sort (1) follows from Observation~\ref{distance_sort_stable} as the distance part is determined.

As a stable orbit cannot become unstable in an extension, to show (2) it is enough to observe that a one-point orbit extension at $x$ does not add any new orbits.
If $x$ is a distance, we only extend an existing orbit of the distance sort, and if $x$ is a point, we extend an existing orbit of the point sort and possibly extend an existing orbit of the distance sort, but by Proposition~\ref{thm:unique_onepoint_extension} no new orbit of the distance sort is added.

To show (3), we inductively close or merge existing orbits as in the proof of Proposition~\ref{lo_totaliz}, but first we take care of all orbits from the point sort, which may introduce new orbits of the distance sort, and then we take care of all orbits of the distance sort, and this does not interfere with the point sort.

Since the assumptions (1), (2), (3) are satisfied, by Proposition~\ref{prop: suffimpldeterm} and Theorem~\ref{thm:amalgamation_from_totalization} we have that $\Ult^*_\fin$ has CAP and we have (a)~$\Leftrightarrow$~(b)~$\Leftarrow$~(c) in the present proposition.
The implication (a)~$\Rightarrow$~(c) follows from Proposition~\ref{thm:stable_necessary}.
\end{proof}

\begin{wn} \label{cor:generic_auto}
    The countable Urysohn ultrametric space $\U$ has a generic dc-automorphism.
\end{wn}
\begin{proof}
    By the Theorem \ref{spaces_totaliz},  $\Ult^*_\fin$ has CAP, and since the empty automorphism $\seq{\emptyset, \emptyset}$ is total and so an amalgamation base, JEP follows from amalgamation.
    Hence, $\U$ has a generic automorphism by the theorem of Kechris and Rosendal (Theorem~\ref{thm:KR}).
\end{proof}

In both $(\Lin^0)^*_\fin$ and $\Ult^*_\fin$ the amalgamation bases are exactly partial automorphisms with stable orbits.
We believe this is not the case in general, but we failed to find a simple counterexample.
\begin{question} \label{que:stable_always_works}
    Is there a nice category of structures $\C$ such that $\C_\fin$ has SAP, the full subcategory of $\C^*$ of total automorphisms has AP, and $\C^*_\fin$ has CAP, but not every $\seq{A, p} \in \C^*_\fin$ with stable orbits is determined?
\end{question}

Coming back to $\Ult^*_\fin$,
let us define a subfamily $\Su' \subseteq \Su$ of all amalgamations bases $\seq{A, p}$ where we additionally suppose that $\seq{A, p}$ is precise, i.e. the spaces $A$, $\dom(p)$, and $\rng(p)$ use all distances, and that every orbit of the point sort is closed or monotone.

\begin{prop} \label{thm:sufficient_cofinal}
    The family $\Su' \subseteq \Su$ is cofinal, i.e. every partial automorphism $\seq{A, p} \in \Ult^*_\fin$ has an extension in $\Su'$.
\end{prop}
\begin{proof}
    Without loss of generality, by Theorem~\ref{spaces_totaliz}, we already have $\seq{A, p} \in \Su$.
    We apply Theorem~\ref{closedormonotone} to extend every orbit of the point sort, one by one, so that it becomes closed or monotone.
    Clearly, every orbit of the point sort of the resulting extension $\seq{A', p'}$ remains stable, but we may have added some new orbits of the distance sort.
    We consider an extension $\seq{E, f} \geq \seq{D_{A'}, D_{p'}}$ in $(\Lin^0_\fin)^*$ where all orbits are stable and without loss of generality $f(r) = r$ for $r = \max(E)$.
    By passing to a suitable orbit extension if needed, we may additionally assume that $E = \dom(f) \cup \rng(f)$ (so every $t \in E$ is a part of an increasing or decreasing or fix-point orbit) and that the end-points of the increasing and decreasing orbits of $f$ are not elements of $D_{\dom(p')}$ and $D_{\rng(p')}$.

    Finally we pass to an extension $\seq{A', p'} \leq \seq{A'', p''} \in \Su'$ with $\seq{D_{A''}, D_{p''}} = \seq{E, f}$.
    For every increasing or decreasing orbit $t = (t_i)_{0 \leq i < k}$ of $\seq{E, f}$ we consider the space $M_t = \set{a_i: 0 \leq i \leq k}$ with the distance $d(a_i, a_j) = t_i$ for $0 \leq i < j \leq k$ for the increasing case and $d(a_i, a_j) = t_{j - 1}$ for the decreasing case, and its partial automorphism $m_t(a_i) = a_{i + 1}$ for $0 \leq i < k$.
    Note that $\dom(m_t)$ uses all the distances $t_i$, but the last one; $\rng(m_t)$ uses all the distance $t_i$, but the first one.
    Next, for every fixed point $x \in E \setminus \set{r}$ we consider the space $F_x = \set{b_0, b_1}$ with the distance $d(b_0, b_1) = x$, and its partial automorphism $q_x(b_i) = b_i$, so that the distance $x$ is used both in $\dom(q_x)$ and $\rng(q_x)$.
    Finally we put the parts together: we let 
    \[\textstyle
        \seq{A'', p''} = \seq{A', p'} \cup \bigcup_t \seq{M_t, m_t} \cup \bigcup_x \seq{F_x, q_x}
    \]
    and let $d(y, z) = r$ for $y$ and $z$  from different parts.

    Clearly, every orbit of the point sort is closed or monotone, and $\seq{D_{A''}, D_{p''}} = \seq{E, f}$, so all orbits of the distance sort are stable.
    By the additional assumptions on $\seq{E, f}$, the spaces $A''$, $\dom(p'')$, $\rng(p'')$ are precise.
    It remains to show that orbits of the point sort are stable.
    As every orbit is already closed or monotone, it is enough to show that the orbits cannot be merged.
    Since we started with $\seq{A', p'}$, no two orbits from $p'$ can be merged.
    Clearly, no new fixed-point orbit can be merged with anything.
    Every extension of a new monotone orbit satisfies $d(y, z) < r$ for each pair of its points since $D_{p''}(r) = r$, and therefore cannot be merged with any other orbit of $\seq{A'', p''}$ as the distance between different orbits is $r$.
\end{proof}

\begin{wn}\label{wn:generic-clo or mon}
    Every orbit of the point sort of the generic dc-automorphism $g$ of $\U$ is closed or monotone, and $D_g$ is the generic automorphism of $D_\U = \Qpos$.
\end{wn}
\begin{proof}
    For both parts we use Theorem~\ref{thm:KR}, which says that the generic automorphism is the limit of the corresponding weak Fraïssé class of partial automorphisms.
    For the first part use Observation~\ref{obs:limit_from_cofinal} to obtain the generic dc-automorphism as a colimit of a sequence from any cofinal subclass of $\Ult_\fin^*$.
    We can take the subclass of amalgamation bases $\Su'$, which is cofinal by Proposition~\ref{thm:sufficient_cofinal}, or even the class of all finite partial dc-automorphisms with only closed or monotone orbits, which is cofinal already by Theorem~\ref{closedormonotone}.

    For the second part can we apply Corollary~\ref{thm:generic_transfer} to the forgetful functor $D\maps \Ult \to \Lin^0$ as $D^*_\fin$ is clearly wide and star-surjective by Observation~\ref{obs:liftings}.
\end{proof}

\section{No generic pair of dc-automorphisms}
\label{sec:no_pair}

In this section we show that the countable Urysohn ultrametric space $\U$ has no generic pair of dc-automorphisms.

The following result was originally proved by Hodkinson in an unpublished note.

\begin{tw}[{Hodkinson, see Truss~\cite[Theorem~2.4]{T}}]\label{Truss}
    There is no generic pair of automorphisms of $\seq{\Q, \leq}$.
    Equivalently, the class $(\Lin_\fin)^{*2}$ does not have WAP.
\end{tw}

Clearly the category $(\Lin_\fin^0)^{*2}$ is equivalent to $(\Lin_\fin)^{*2}$, therefore $(\Lin_\fin^0)^{*2}$ does not have WAP as well.

\begin{wn} \label{cor:no_generic_pair}
    The class $\Ult_\fin^{*2}$ does not have WAP, and so $\aut(\U)$ does not have a generic pair.
\end{wn}
\begin{proof}
    We will use the contraposition of the transfer principle from Proposition~\ref{thm:fin_transfer}(3) applied to the forgetful functor $(D_\fin)^{*2}\maps (\Ult_\fin)^{*2} \to (\Lin^0_\fin)^{*2}$.
    Observation~\ref{obs:liftings} implies that the functor $(D_\fin)^{*2}$ is star-surjective.
    Clearly, for every $\seq{E, \phi_1, \phi_2} \in (\Lin^0)^{*2}$ we may consider $X = \seq{\emptyset, E}$ and $p_i = \seq{\emptyset, \phi_i}$ for $i = 1, 2$, showing that $(D_\fin)^{*2}$ is wide.
    Hence, $(\Ult_\fin)^{*2}$ does not have WAP, since by Theorem~\ref{Truss}, $(\Lin^0_\fin)^{*2}$ does not have WAP.
    Therefore, there is no generic pair of dc-automorphisms of $\U$ by Theorem~\ref{thm:KR}
\end{proof}

In the previous paper~\cite[Section~3.3]{paper1} we studied the category $\UltConv_\fin$ of finite convexly ordered ultrametric spaces.
We showed that it is Fraïssé and that its limit $\U^\prec$ is the rational Urysohn ultrametric space $\U$ endowed with a generic convex order~\cite[Theorem~3.16]{paper1}.
The distance functor $(D^\prec_\fin)^{*2}\maps (\UltConv_\fin)^{*2} \to (\Lin^0)^{*2}$ is again wide and star-surjective by an analogous argument, and so we obtain the following.

\begin{wn} \label{cor:no_generic_pair_convex}
    The class $(\UltConv_\fin)^{*2}$ does not have WAP, and so $\aut(\U^\prec)$ does not have a generic pair.
\end{wn}

\begin{question} \label{que:generic_auto_convex}
Does  $\Aut(\U^\prec)$ have a generic dc-automorphism?
\end{question}

We believe that our proof for $\Aut(\U)$  should extend to $\Aut(\U^\prec)$.


\end{document}